\documentclass[11pt]{article}

\usepackage[a4paper,margin=1in]{geometry}
\usepackage{amsmath,amssymb,amsfonts,amsthm,mathtools}
\usepackage{algorithm}
\usepackage{algpseudocode}
\usepackage{enumitem}
\usepackage{hyperref}
\usepackage[nameinlink,capitalize]{cleveref}
\hypersetup{colorlinks=true,linkcolor=blue,citecolor=blue,urlcolor=blue}
\usepackage[numbers,sort&compress]{natbib}

\newcommand{\R}{\mathbb R}
\newcommand{\ip}[2]{\left\langle #1,#2\right\rangle}
\newcommand{\norm}[1]{\left\|#1\right\|}
\newcommand{\cL}{\mathcal L}
\newcommand{\cG}{\mathcal G}
\newcommand{\cE}{\mathcal E}
\newcommand{\cS}{\mathcal S}

\newcommand{\bigO}{\mathcal O}

\theoremstyle{plain}
\newtheorem{theorem}{Theorem}[section]
\newtheorem{lemma}[theorem]{Lemma}
\newtheorem{proposition}[theorem]{Proposition}

\theoremstyle{definition}
\newtheorem{assumption}[theorem]{Assumption}
\newtheorem{remark}[theorem]{Remark}

\numberwithin{equation}{section}

\title{Fast primal-dual methods for convex-concave bilinear saddle point problems: continuous-time dynamics and discrete algorithms}
\author{Xin He\thanks{School of Science, Xihua University, Chengdu, China. 
Email: \texttt{hexinuser@163.com}.}, Ya-Ping Fang\thanks{Department of Mathematics, Sichuan University, Chengdu, China
Email: \texttt{ypfang@scu.edu.cn}}}
\date{}

\begin{document}
\maketitle

\begin{abstract}
This paper studies Nesterov accelerated methods for continuously differentiable convex-concave bilinear saddle point problems. For the continuous-time model, we analyze a second-order primal-dual dynamical system with vanishing damping $\alpha/t$, where $\alpha\geq 3$. Under the merely convex-concave setting, we prove convergence of the primal-dual trajectory to a saddle point. In the noncritical regime $\alpha>3$, we further obtain the improved rate $o(1/t^{2})$ for the primal-dual gap and $o(1/t)$ for the velocity, and, under an additional Lipschitz gradient assumption, $o(1/t)$ for the stationarity residual. We then derive a structure-preserving finite-difference discretization, which leads to a fast primal-dual algorithm with Nesterov extrapolation. For a general accelerated parameter sequence ${t_k}$ satisfying $t_{k+1}^2-t_k^2\le \rho t_{k+1}$ with $\rho\in(0,1]$, we prove the $\bigO(1/t_k^{2})$  convergence rate for the primal-dual gap and convergence of the  generated sequence. In the noncritical case $\rho<1$, we further establish the improved rate $o(1/t_k^{2})$ for the gap and $o(1/t_k)$ for the stationarity residual. These results provide continuous-discrete acceleration methods for bilinear saddle point problems in the merely convex-concave setting.
\end{abstract}
 

\medskip

  \section{Introduction} 
 
Bilinear convex-concave saddle point problems arise naturally in variational inequalities, imaging, signal processing, machine learning, and constrained convex programming; see, for example, \cite{ChambolleJMIV,DuICML,BotCOAP,Wangips,ChambolleMP}. In this paper, we consider 
\begin{equation}\label{prob:saddle} 
   \min_{x\in\mathbb R^n}\max_{y\in\mathbb R^m} \ \mathcal L(x,y):=f(x)+\langle Kx,y\rangle-g(y), 
\end{equation}
where $f:\mathbb R^n\to\mathbb R$ and $g:\mathbb R^m\to\mathbb R$ are convex and continuously differentiable functions, and $K\in\mathbb R^{m\times n}$ is a given matrix. This model is also closely related to the composite convex minimization problem 
\[ \min_{x\in\mathbb R^n}\ f(x)+g^*(Kx),\]
where $g^*$ denotes the Fenchel conjugate of $g$. Hence, \eqref{prob:saddle} provides a unified formulation for many structured convex optimization problems with a linear coupling.

 First-order primal-dual methods are among the most important tools for solving \eqref{prob:saddle}. The primal-dual algorithms of Chambolle and Pock \cite{ChambolleJMIV} and related splitting methods provide efficient frameworks for convex-concave saddle point problems with linear coupling. Extragradient, mirror-prox, forward-backward-forward, optimistic, and reflected-gradient methods use correction or prediction steps to control the skew-symmetric interaction generated by the bilinear term; see, for example, \cite{KorpelevichEMM,NemirovskiSIOPT,TsengSICON,PopovMN,MalitskySIOPT20,MokhtariSIOPT}. In the general convex-concave setting, these methods typically provide $\bigO(1/k)$ convergence guarantees for primal-dual gaps. Faster rates can be obtained under stronger assumptions, such as strong convexity, partial strong convexity, or special metric regularity. Accelerated primal-dual methods for strongly convex-concave and strongly convex-strongly concave problems have been developed in \cite{KhalafiPMLR,BotCOAP,Tran20,ChangJSC,Tran22,HeCNSNS,HeStrong,CondatPre2026}. However, in the merely convex-concave bilinear setting of \eqref{prob:saddle}, it is still difficult to obtain accelerated rates together with last-iterate convergence and improved asymptotic estimates. This difficulty is caused by the simultaneous presence of Nesterov-type inertia and the skew-symmetric primal-dual coupling.

 A continuous-time dynamic viewpoint offers a useful way to understand the mechanism of acceleration. For unconstrained convex minimization, Su et al. \cite{SuJMLR} showed that Nesterov's accelerated gradient method \cite{Nesterov18,Nesterov83} is closely related to the second-order vanishing damping system 
\begin{equation}\label{eq:avd} 
	\ddot x(t)+\frac{\alpha}{t}\dot x(t)+\nabla f(x(t))=0. 
\end{equation} 
When $\alpha\ge3$, this dynamic yields the accelerated rate $f(x(t))-\min f=\bigO(1/t^{2})$. When $\alpha>3$, trajectory convergence and the improved estimate $f(x(t))-\min f=o(1/t^{2})$ have been established in \cite{AttouchMp,MayTJM}. The convergence of the trajectory of \eqref{eq:avd} in the critical case $\alpha=3$ has been investigated in \cite{Jang} by an AI-assisted approach. Based on time discretizations of \eqref{eq:avd} and corresponding Lyapunov functions, the convergence of iterates and convergence rates of Nesterov-type accelerated algorithms have been studied in \cite{WibisonoPans,AttouchSIOPT,Attouch18siam,Jang,Bot25}. This continuous-time viewpoint has also been extended to primal-dual frameworks for linearly constrained convex optimization. Continuous-time primal-dual dynamics with vanishing damping $\alpha/t$ have been studied in \cite{ZhaoJMLR,HeNN,HeTc,ZengTAC,BotJDE,HeSIAM}, while Nesterov-type accelerated primal-dual algorithms have been developed in \cite{HeNA,LuoMC,HePDarxiv,BotMP,LuoJoGo}. These works show that Lyapunov analysis for continuous-time dynamics is not only a tool for deriving fast rates, but also provides useful guidance for designing discrete accelerated algorithms with matching convergence properties.

Extending the Nesterov's accelerated dynamic \eqref{eq:avd} from unconstrained optimization to saddle point problems is nontrivial. A direct insertion of the skew-symmetric coupling into the inertial system does not automatically produce a useful Lyapunov estimate, because the rotational component generated by the bilinear term must be controlled together with the damping and extrapolation terms. For bilinear saddle point problems, Zeng et al. \cite{ZengIFAC} proposed a Nesterov's accelerated continuous-time dynamic with correction extrapolation terms and obtained an $\bigO(1/t^{2})$ convergence rate. He et al. \cite{HeAMO} studied a more general second-order primal-dual dynamical system with general damping and scaling parameters, and established convergence rates for primal-dual gaps and velocities. Ding et al. \cite{DingCOAP} proposed an accelerated primal-dual algorithm for bilinearly coupled saddle point problems based on a discretization of a continuous dynamic, and proved $\bigO(1/k^2)$ primal-dual gap estimates together with convergence of the generated iterates under convex-concave assumptions. These results show that Nesterov acceleration can be adapted to bilinear saddle structures. Nevertheless, a complete continuous-discrete theory is still missing. In particular, one would like to obtain trajectory convergence and improved $o(1/t^{2})$ estimates for the continuous dynamic, and at the same time derive a discrete algorithm whose last-iterate convergence and asymptotic rates are consistent with the continuous-time model.

 The aim of this paper is to develop such a continuous-discrete framework for \eqref{prob:saddle} under merely convex-concave assumptions. The framework is designed to capture the same cancellation mechanism in both the continuous and discrete settings. By constructing a continuous Lyapunov function and a discrete energy sequence, we establish accelerated convergence rates, full convergence, and improved little-$o$ asymptotic estimates for the proposed dynamic and its associated accelerated primal-dual algorithm. Our main contributions are summarized as follows.

{\bf (a) Continuous-time analysis.} We study a Nesterov-type inertial primal-dual dynamical system with vanishing damping $\alpha/t$ for the bilinear saddle point problem \eqref{prob:saddle}. Related dynamics have been considered in \cite{ZengIFAC,HeAMO,DingCOAP}, where boundedness of trajectories and $\bigO(1/t^{2})$ primal-dual gap rates were obtained. In this paper, we go beyond these estimates by proving convergence of the whole primal-dual trajectory to a saddle point under only the convexity and continuous differentiability of $f$ and $g$. In the noncritical case $\alpha>3$, we establish the improved estimates 
\[ 
 	\mathcal L(x(t),y^*)-\mathcal L(x^*,y(t))=o\left(\frac{1}{t^{2}}\right), \qquad \|(\dot x(t),\dot y(t))\|=o\left(\frac{1}{t}\right). 
\] 
When $f$ and $g$ have Lipschitz continuous gradients, we also prove an $o(1/t)$ rate for the stationarity residual. To the best of our knowledge, trajectory convergence and such little-$o$ estimates for Nesterov-type saddle dynamics have not previously been established in this general convex-concave setting. These results extend the asymptotic convergence theory for Nesterov-type dynamics \cite{HeTc,MayTJM,Jang,AttouchMp,HeSIAM,BotJDE} from unconstrained and linearly constrained optimization to bilinear saddle point problems.

{\bf (b) Discrete-time analysis.} We derive a fast primal-dual algorithm by a structure-preserving discretization of the continuous dynamic. The scheme combines Nesterov extrapolation with a semi-implicit bilinear correction, leading to a discrete Lyapunov estimate. It can be implemented by solving one strongly convex quadratic primal subproblem followed by one explicit dual update. Under the Lipschitz continuity of $\nabla f$ and $\nabla g$, we prove the accelerated saddle gap estimate
$
	\mathcal L(x_k,y^*)-\mathcal L(x^*,y_k)=\bigO\left({1}/{t_k^{2}}\right), 
$
which gives an $\bigO(1/k^{2})$ rate for standard choices of $\{t_k\}$. Compared with classical primal-dual splitting type methods \cite{KorpelevichEMM,NemirovskiSIOPT,TsengSICON,PopovMN,MokhtariSIOPT}, which usually yield $\bigO(1/k)$ guarantees in the general convex-concave case, the proposed method achieves a faster rate. Moreover, compared with the discrete algorithms in \cite{DingCOAP}, we further prove convergence of the whole generated sequence under a critical Nesterov-type parameter rule, and obtain $o(1/t_k^{2})$ estimates for the primal-dual gap and $o(1/t_k)$ estimates for the stationarity residual in the noncritical regime.

\subsection{Organization}

The rest of the paper is organized as follows. Section \ref{sec:prelim} collects the notation and preliminary facts used throughout the paper. Section \ref{sec:continuous} studies the continuous-time model, proves the Lyapunov estimate, establishes trajectory convergence, and derives the little-$o$ rates. Section \ref{sec:discrete} derives the discrete algorithm and proves the corresponding discrete Lyapunov estimate and convergence results. Section \ref{sec:numerics} reports numerical experiments. Section \ref{sec:Conclusion} concludes the paper.

\section{Preliminaries}\label{sec:prelim}
 
In this section, we collect the notation and several elementary facts that will be used in the subsequent analysis. Denote by $\cS$ the saddle point set of \eqref{prob:saddle}. For any $z^*=(x^*,y^*)\in\cS$, the saddle point inequality holds:
\begin{equation}\label{eq:saddle-ineq}
    \cL(x^*,y)\le \cL(x^*,y^*)\le \cL(x,y^*),
    \qquad
    \forall (x,y)\in\R^n\times\R^m.
\end{equation}
Throughout this paper, we assume that $\cS\ne\emptyset$. Since $f$ and $g$ are differentiable, the saddle point condition is equivalent to the following first-order system:
\begin{equation}\label{eq:kkt-saddle}
    \nabla f(x^*)+K^\top y^*=0,
    \qquad
    \nabla g(y^*)-Kx^*=0.
\end{equation}

To state convergence rates for both the continuous trajectory and the discrete iterates, we use two standard measures of optimality. The first one is the primal-dual gap with respect to a fixed saddle point. More precisely, for $z^*=(x^*,y^*)\in\cS$, define
\[
    \cG_{z^*}(z):=\cL(x,y^*)-\cL(x^*,y),
    \qquad z=(x,y)\in\R^n\times\R^m.
\]
By \eqref{eq:saddle-ineq}, we have
$ \cG_{z^*}(z)\ge0$  for any  $z\in\R^n\times\R^m$ and $\cG_{z^*}(z^*)=0$.

The second measure is the stationarity residual associated with the first-order saddle system. It will be used later to quantify the violation of stationarity. Set
\begin{equation}\label{def:zhS} 
	 z=(x,y)\in\R^n\times\R^m, \qquad h(z):=f(x)+g(y), \qquad S:= \begin{pmatrix} 0&K^\top\\ -K&0 \end{pmatrix}. 
\end{equation}
Then $h$ is convex and differentiable, while $S$ is skew-symmetric, namely $S^\top=-S$. Hence
\[
 	\ip{Sz}{z}=0, \qquad \forall z\in\R^n\times\R^m.
\]
We define the stationarity residual by
\[
  F(z):=\nabla h(z)+Sz=\begin{pmatrix} \nabla f(x)+K^\top y\\ \nabla g(y)-Kx \end{pmatrix}. 
\]
Therefore,
\begin{equation}\label{eq:eqvF} 
 		z^*=(x^*,y^*)\in\cS \quad\Longleftrightarrow\quad F(z^*)=0. 
\end{equation}
The quantity $\|F(z)\|$ measures the violation of the first-order saddle system \eqref{eq:kkt-saddle}. Thus, the primal-dual gap $\cG_{z^*}(z)$ and the stationarity residual $\|F(z)\|$ provide two complementary descriptions of convergence. The operator $F$ is monotone. Indeed, for any $z_i=(x_i,y_i)\in\R^n\times\R^m$, $i=1,2$, the skew-symmetry of $S$ gives
\begin{equation}\label{eq:Fmono}
\begin{aligned}
    &\ip{F(z_1)-F(z_2)}{z_1-z_2} =
    \ip{\nabla h(z_1)-\nabla h(z_2)}{z_1-z_2}\\
   & \qquad=\ip{\nabla f(x_1)-\nabla f(x_2)}{x_1-x_2}
  +\ip{\nabla g(y_1)-\nabla g(y_2)}{y_1-y_2}\\
      &\qquad \ge 0,
\end{aligned}
\end{equation}
where the last inequality follows from the convexity of $f$ and $g$.

We next record a useful expression for the primal-dual gap. Since $z^*\in\cS$, \eqref{eq:eqvF} gives $Sz^*=-\nabla h(z^*)$. Thus, for any $z=(x,y)$,
\begin{equation}\label{eq:gap} 
 \begin{aligned}
 	 \cG_{z^*}(z) &= \cL(x,y^*)-\cL(x^*,y)\\ 
 	 	&= f(x)-f(x^*)+g(y)-g(y^*)+\ip{Kx}{y^*}-\ip{Kx^*}{y}\\ 
 	 	&= h(z)-h(z^*)-\ip{\nabla h(z^*)}{z-z^*}. 
 	\end{aligned} 
\end{equation}

The following structural property of the saddle point set will be used in the convergence analysis.
\begin{proposition}\label{prop:str} 
 Consider \eqref{prob:saddle}, where $f:\R^n\to\R$ and $g:\R^m\to\R$ are convex and differentiable. Let $z_i=(x_i,y_i)\in\cS$, $i=1,2$. Then
 \begin{itemize}
 	\item [(i)] $\nabla f(x_1)=\nabla f(x_2)$ and $\nabla g(y_1)=\nabla g(y_2)$;
 	\item [(ii)] $Kx_1=Kx_2$ and $K^\top y_1=K^\top y_2$;
 	\item [(iii)] $\cG_{z_1}(z)=\cG_{z_2}(z)$  for any $z\in\R^n\times\R^m$.
 \end{itemize}
\end{proposition}
\begin{proof}
Since $z_i\in\cS$, we have $F(z_i)=0$ for $i=1,2$. Hence, by \eqref{eq:Fmono},
\[ 
\begin{aligned}
0 &= \ip{F(z_1)-F(z_2)}{z_1-z_2}\\
  &= \ip{\nabla f(x_1)-\nabla f(x_2)}{x_1-x_2}
     +\ip{\nabla g(y_1)-\nabla g(y_2)}{y_1-y_2}.
\end{aligned}
\]
By the convexity of $f$ and $g$, both terms on the right-hand side are nonnegative. Therefore,
\begin{equation}\label{eq:str0}
   \ip{\nabla f(x_1)-\nabla f(x_2)}{x_1-x_2}=0, \qquad
   \ip{\nabla g(y_1)-\nabla g(y_2)}{y_1-y_2}=0.
\end{equation}

We first prove that $\nabla f(x_1)=\nabla f(x_2)$. Set
$
   A_f:=f(x_1)-f(x_2)-\ip{\nabla f(x_2)}{x_1-x_2}
$
and
$
   B_f:=f(x_2)-f(x_1)-\ip{\nabla f(x_1)}{x_2-x_1}.
$
By the convexity of $f$, we have $A_f\ge0$ and $B_f\ge0$. Moreover, \eqref{eq:str0} implies
\[
   A_f+B_f=\ip{\nabla f(x_1)-\nabla f(x_2)}{x_1-x_2}=0.
\]
Thus $A_f=B_f=0$. In particular,
\begin{equation}\label{eq_f0} 
   		f(x_1)=f(x_2)+\ip{\nabla f(x_2)}{x_1-x_2}.
\end{equation}
On the other hand, the convexity of $f$ gives
$
   f(x)\ge f(x_2)+\ip{\nabla f(x_2)}{x-x_2},
$ for any $x\in\R^n$.
Combining this inequality with \eqref{eq_f0}, we obtain
\[
   f(x)\ge f(x_1)+\ip{\nabla f(x_2)}{x-x_1},
   \qquad \forall x\in\R^n.
\]
Hence $\nabla f(x_2)\in\partial f(x_1)=\{\nabla f(x_1)\}$. Therefore, $\nabla f(x_1)=\nabla f(x_2)$. Applying the same argument to $g$ and using the second equality in \eqref{eq:str0}, we obtain
\begin{equation}\label{eq_g0} 
   g(y_1)=g(y_2)+\ip{\nabla g(y_2)}{y_1-y_2},
\end{equation}
and $\nabla g(y_1)=\nabla g(y_2)$. This proves (i).

Next, since $z_i\in\cS$, the saddle point conditions give
\[
   \nabla f(x_i)+K^\top y_i=0, \qquad
   \nabla g(y_i)-Kx_i=0, \qquad i=1,2.
\]
Using (i) in these identities yields (ii).

It remains to prove (iii). From (i), we have $\nabla h(z_1)=\nabla h(z_2)$. By \eqref{eq:gap}, for any $z\in\R^n\times\R^m$,
\begin{align*} 
 \cG_{z_1}(z)-\cG_{z_2}(z)
 &= h(z_2)-h(z_1)-\ip{\nabla h(z_1)}{z-z_1}
    +\ip{\nabla h(z_2)}{z-z_2} \\ 
 &= h(z_2)-h(z_1)+\ip{\nabla h(z_2)}{z_1-z_2} \\ 
 &=0,
\end{align*}
where the last equality follows from \eqref{eq_f0} and \eqref{eq_g0}. Hence (iii) holds.
\end{proof}

 We shall also use the following smoothness assumption when deriving estimates for the stationarity residual and for the discrete algorithm.

\begin{assumption}\label{ass:disc}
 The functions $f:\R^n\to\R$ and $g:\R^m\to\R$ are convex and differentiable. Moreover, $\nabla f$ and $\nabla g$ are Lipschitz continuous with constants $L_f>0$ and $L_g>0$, respectively; that is, for any $x,\tilde x\in\R^n$ and $y,\tilde y\in\R^m$,
 \[
    \|\nabla f(x)-\nabla f(\tilde x)\| \le L_f\|x-\tilde x\|,
    \qquad
    \|\nabla g(y)-\nabla g(\tilde y)\| \le L_g\|y-\tilde y\|.
 \] 
\end{assumption}

Let $r,s>0$ and define
\begin{equation}\label{def:Ddisc} 
D:= \begin{pmatrix} r^{-1}I_n&0\\ 0&s^{-1}I_m \end{pmatrix}, \qquad
\|z\|_D^2:=\ip{Dz}{z}, \qquad z=(x,y)\in\R^n\times\R^m.
\end{equation}
Then $D$ is a positive definite matrix.

\begin{lemma}\label{lem:D-smooth}
 Let  \Cref{ass:disc} hold and suppose that $0<r\le \frac1{L_f}$ and $0<s\le \frac1{L_g}$. Then, for all $u=(x,y)\in\R^n\times\R^m$ and $v=(\tilde x,\tilde y)\in\R^n\times\R^m$,
 \begin{itemize}
 	\item [(i)] $h(u)\le h(v)+\ip{\nabla h(v)}{u-v}+\frac12\|u-v\|_D^2$;
 	\item [(ii)] $\|\nabla h(u)-\nabla h(v)\|^2\le 2\max\{L_f,L_g\}\bigl(h(u)-h(v)-\ip{\nabla h(v)}{u-v}\bigr)$.
 \end{itemize}
Here $h$ and $D$ are defined in \eqref{def:zhS} and \eqref{def:Ddisc}, respectively.
\end{lemma}

\begin{proof}
Let $u=(x,y)$ and $v=(\tilde x,\tilde y)$. Since $\nabla f$ and $\nabla g$ are Lipschitz continuous, it follows from \cite[Theorem 2.1.5]{Nesterov18} and $0<r\le \frac1{L_f}$, $0<s\le \frac1{L_g}$ that
\[
\begin{aligned}
h(u)&=f(x)+g(y)\\
&\le f(\tilde x)+\ip{\nabla f(\tilde x)}{x-\tilde x}
   +\frac{L_f}{2}\|x-\tilde x\|^2
   +g(\tilde y)+\ip{\nabla g(\tilde y)}{y-\tilde y}
   +\frac{L_g}{2}\|y-\tilde y\|^2\\
&\le h(v)+\ip{\nabla h(v)}{u-v}
   +\frac{1}{2r}\|x-\tilde x\|^2
   +\frac{1}{2s}\|y-\tilde y\|^2\\
&= h(v)+\ip{\nabla h(v)}{u-v}+\frac12\|u-v\|_D^2.
\end{aligned}
\]
This proves (i).

Since $f$ and $g$ are convex with Lipschitz continuous gradients, it also follows from \cite[Theorem 2.1.5]{Nesterov18} that
\[
\begin{aligned}
& \frac1{L_f}\|\nabla f(x)-\nabla f(\tilde x)\|^2
 + \frac1{L_g}\|\nabla g(y)-\nabla g(\tilde y)\|^2\\
&\qquad\le
2\bigl(f(x)-f(\tilde x)-\ip{\nabla f(\tilde x)}{x-\tilde x}
+g(y)-g(\tilde y)-\ip{\nabla g(\tilde y)}{y-\tilde y}\bigr)\\
&\qquad=
2\bigl(h(u)-h(v)-\ip{\nabla h(v)}{u-v}\bigr).
\end{aligned}
\]
This together with $h(z)=f(x)+g(y)$ implies (ii).
\end{proof}

We conclude this section with several auxiliary lemmas.

\begin{lemma}\cite[Lemma A.2]{HePDarxiv}\label{le_wbound} 
Let $D\in\mathbb R^{d\times d}$ be symmetric positive definite, $\eta>0$, and $\{a_k\}_{k\ge1}$ be a nonnegative sequence. Let $\{z_k\}_{k\ge0}\subset\mathbb R^d$ with $z_1=z_0$, and fix $z^*\in\mathbb R^d$. Suppose that 
$
   \sup_{k\ge1} \|\eta(z_k-z^*)+a_k(z_k-z_{k-1})\|_D \le C
$
for some $C\ge0$. Then
\[
	\|z_k-z^*\|_D\le \frac{C}{\eta}, \qquad
	a_k\|z_k-z_{k-1}\|_D\le 2C, \qquad
	\forall k\ge1.
\]
\end{lemma}

\begin{lemma}\cite[Proposition~A.30]{Bertsekas} \label{le_bertsekas}
Let $\{\omega_k\}_{k\ge1}$, $\{\sigma_k\}_{k\ge1}$, and $\{q_k\}_{k\ge1}$ be nonnegative sequences satisfying
\[
    \omega_{k+1}\le (1-\sigma_k)\omega_k+\sigma_k q_k,
    \qquad 0< \sigma_k\le 1,
    \qquad \sum_{k=1}^{+\infty}\sigma_k=+\infty.
\]
If $q_k\to0$ as $k\to+\infty$, then $\omega_k\to0$.
\end{lemma}

\section{Convergence analysis for continuous-time primal-dual dynamic}\label{sec:continuous}

In this section, we further study the continuous-time primal-dual dynamical system associated with the bilinear saddle point problem \eqref{prob:saddle}. The following Nesterov accelerated primal-dual dynamical system has been introduced in \cite{ZengIFAC,HeAMO,DingCOAP}: 
\begin{equation}\label{dyn:cont-xy}
\begin{cases}
\ddot x(t)+\dfrac{\alpha}{t}\dot x(t)+\nabla f(x(t))+K^\top\bigl(y(t)+\theta t\dot y(t)\bigr)=0,\\[1.2ex]
\ddot y(t)+\dfrac{\alpha}{t}\dot y(t)+\nabla g(y(t))-K\bigl(x(t)+\theta t\dot x(t)\bigr)=0.
\end{cases}
\end{equation}
 Here $t\ge t_0>0$, $\alpha\ge3$ and $\theta\in[\frac1{\alpha-1},\frac12]$. The works \cite{ZengIFAC,HeAMO,DingCOAP} established accelerated $\bigO(t^{-2})$ estimates for the primal-dual gap, together with related boundedness and velocity estimates. However, the convergence of the whole trajectory and the improved little-$o$ convergence rates were not investigated there. Our purpose in this section is to prove these additional asymptotic properties.

For the subsequent analysis, it is convenient to rewrite \eqref{dyn:cont-xy} in a compact form. By the notation introduced in \eqref{def:zhS}, dynamic \eqref{dyn:cont-xy} is equivalent to 
\begin{equation}\label{dyn:cont-z}
    \ddot z(t)+\frac{\alpha}{t}\dot z(t)+\nabla h(z(t))+S\bigl(z(t)+\theta t\dot z(t)\bigr)=0,
    \qquad t\ge t_0>0.
\end{equation}
 This compact formulation separates the convex gradient part $\nabla h$ from the skew-symmetric coupling part $S$, and hence simplifies the Lyapunov analysis below.

Let $z(t)=(x(t),y(t))$ be a trajectory of dynamic \eqref{dyn:cont-z}. For a fixed saddle point $z^*=(x^*,y^*)\in\cS$, define the continuous Lyapunov energy function
\begin{equation}\label{def:E-cont}
    \cE_{z^*}(t)
    :=
    \theta^2t^2\cG_{z^*}(z(t))
    +\frac12\norm{v_{z^*}(t)}^2
    +\frac{\xi}{2}\norm{z(t)-z^*}^2,
\end{equation}
where
$
    \xi:=\theta\alpha-\theta-1$ and $
    v_{z^*}(t):=z(t)-z^*+\theta t\dot z(t).
$
Since $\alpha\ge3$ and $\theta\in[\frac1{\alpha-1},\frac12]$, we have $\xi\ge0$. Hence $\cE_{z^*}$ is nonnegative because $\cG_{z^*}\ge0$.

The next theorem recalls the basic estimates for dynamic \eqref{dyn:cont-xy}. These estimates follow from the Lyapunov computation used in \cite{ZengIFAC,HeAMO,DingCOAP}. At the endpoint $\theta=\frac1{\alpha-1}$, one has $\xi=0$, but the bound \[ \norm{z(t)-z^*+\theta t\dot z(t)} = \norm{v_{z^*}(t)} \le \sqrt{2\cE_{z^*}(t_0)} \] still implies boundedness of the trajectory and $\norm{\dot z(t)}=\bigO(1/t)$ by \cite[Lemma 1]{HeNN}. Thus the estimates below hold on the whole admissible range $\theta\in\left[\frac1{\alpha-1},\frac12\right]$.

\begin{theorem}\cite{HeAMO,ZengIFAC,DingCOAP}\label{thm:cont-est}
Let $z(t)=(x(t),y(t))$ be a trajectory of dynamic \eqref{dyn:cont-xy}. Then, for any $z^*=(x^*,y^*)\in\cS$, the energy $\cE_{z^*}$ is nonincreasing. Moreover, the trajectory $(x(t),y(t))$ is bounded, and the velocity and the primal-dual gap satisfy
\[
    \norm{(\dot x(t),\dot y(t))}=O\left(\frac1t\right),
    \qquad
    \cL(x(t),y^*)-\cL(x^*,y(t))=O\left(\frac1{t^2}\right).
\]
Furthermore, if $\alpha>3$ and $\theta\in\left(\frac{1}{\alpha-1},\frac{1}{2}\right)$, then
\[
 \int_{t_0}^{+\infty}t(\cL(x(t),y^*)-\cL(x^*,y(t)))\,dt<+\infty,
 \qquad
 \int_{t_0}^{+\infty}t\norm{\dot z(t)}^2\,dt<+\infty.
\]
\end{theorem}

We now turn to the asymptotic behavior of the trajectory. The gap estimate in  \Cref{thm:cont-est} gives $\cG_{z^*}(z(t))\to0$, but this alone does not immediately imply convergence of $z(t)$. The next lemma proves that every cluster point of $z(t)$ belongs to $\cS$.

\begin{lemma}\label{lem:cont-cluster}
Let $z(t)=(x(t),y(t))$ be a trajectory of dynamic \eqref{dyn:cont-z}. Then every cluster point of $z(t)$ belongs to $\cS$.
\end{lemma}

\begin{proof}
Fix an arbitrary point $\hat z=(\hat x,\hat y)\in\cS$. By \Cref{thm:cont-est}, the trajectory $z(t)$ is bounded and
$
    \cG_{\hat z}(z(t))=\bigO\left(1/{t^2}\right).
$
Hence $\cG_{\hat z}(z(t))\to0$.

Let $\bar z$ be any cluster point of $z(t)$. Then there exists a sequence $t_j\to+\infty$ such that $z(t_j)\to \bar z$. Passing to the limit in $\cG_{\hat z}(z(t_j))\to0$ and using \eqref{eq:gap} and the continuous differentiability of $h$, we obtain
$
   h(\bar z)-h(\hat z)-\ip{\nabla h(\hat z)}{\bar z-\hat z}=0.
$
Since $h$ is convex, for every $z\in\R^n\times\R^m$,
$
   h(z)\ge h(\hat z)+\ip{\nabla h(\hat z)}{z-\hat z}.
$
Combining the last two relations gives
\[
    h(z)\ge h(\bar z)+\ip{\nabla h(\hat z)}{z-\bar z},
    \qquad \forall z\in\R^n\times\R^m.
\]
Thus $\nabla h(\hat z)\in\partial h(\bar z)$. Since $h$ is differentiable, it follows that
$
    \nabla h(\bar z)=\nabla h(\hat z).
$

We next show that the whole gradient trajectory converges to $\nabla h(\hat z)$. Suppose, to the contrary, that there exist $\varepsilon_0>0$ and a sequence $s_j\to+\infty$ such that
$
    \norm{\nabla h(z(s_j))-\nabla h(\hat z)}\ge\varepsilon_0.
$
By the boundedness of $z(t)$, after passing to a subsequence if necessary, we may assume that $z(s_j)\to \tilde z$ for some cluster point $\tilde z$. The previous argument gives $\nabla h(\tilde z)=\nabla h(\hat z)$. By the continuity of $\nabla h$,
$
    \nabla h(z(s_j))\to \nabla h(\tilde z)=\nabla h(\hat z),
$
which is a contradiction. Therefore,
\begin{equation}\label{eq:gradhh}
    \nabla h(z(t))\to\nabla h(\hat z).
\end{equation}

It remains to identify the skew-symmetric part. Set
\begin{equation}\label{def:qa}
	 q(t):=S(z(t)-\hat z),
    \qquad
    a(t):=\nabla h(z(t))-\nabla h(\hat z).
\end{equation}
Then $a(t)\to0$ by \eqref{eq:gradhh}. Since $\hat z\in\cS$, from \eqref{eq:eqvF} we have
$
    F(\hat z)=\nabla h(\hat z)+S\hat z=0.
$
Therefore, dynamic \eqref{dyn:cont-z} can be rewritten as
\[
    \ddot z(t)+\frac{\alpha}{t}\dot z(t)
    +a(t)+q(t)+\theta t\dot q(t)=0.
\]
Equivalently,
\begin{equation}\label{eq:qaux}
     t\dot q(t)+\frac{1}{\theta}q(t)
    =
    -\frac{1}{\theta}
    \left(a(t)+\ddot z(t)+\frac{\alpha}{t}\dot z(t)\right).
\end{equation}
Let $p=\frac1\theta$. Since $\theta\le1/2$, we have $p\ge2$. Multiplying \eqref{eq:qaux} by $t^{p-1}$ and integrating over $[t_0,t]$ gives
\begin{align}
q(t)
=&\left(\frac{t_0}{t}\right)^p q(t_0)
  -pt^{-p}\int_{t_0}^{t}s^{p-1}a(s)\,ds \notag\\
&-pt^{-p}\int_{t_0}^{t}s^{p-1}\ddot z(s)\,ds
  -p\alpha t^{-p}\int_{t_0}^{t}s^{p-2}\dot z(s)\,ds .
  \label{eq:qrep}
\end{align}
The first term tends to zero. Since $a(t)\to0$, L'H\^opital's rule yields
\[
\begin{aligned}
\lim_{t\to+\infty}
\left\|t^{-p}\int_{t_0}^{t}s^{p-1}a(s)\,ds\right\|
&\le
\lim_{t\to+\infty}
t^{-p}\int_{t_0}^{t}s^{p-1}\|a(s)\|\,ds =
\lim_{t\to+\infty}\frac{\|a(t)\|}{p}
=0 .
\end{aligned}
\]
For the term involving $\ddot z$, integration by parts gives
\[\begin{aligned}
t^{-p}\int_{t_0}^{t}s^{p-1}\ddot z(s)\,ds
&=
t^{-1}\dot z(t)-t^{-p}t_0^{p-1}\dot z(t_0)
-(p-1)t^{-p}\int_{t_0}^{t}s^{p-2}\dot z(s)\,ds .
\end{aligned}
\]
By \Cref{thm:cont-est}, $\norm{\dot z(t)}=O(1/t)$. Hence
$
    t^{-1}\dot z(t)\to0$ and
$    t^{-p}t_0^{p-1}\dot z(t_0)\to0$.
Moreover, another application of L'H\^opital's rule gives
\[
\begin{aligned}
\lim_{t\to+\infty}
\left\|t^{-p}\int_{t_0}^{t}s^{p-2}\dot z(s)\,ds\right\|
&\le
\lim_{t\to+\infty}t^{-p}\int_{t_0}^{t}s^{p-2}\|\dot z(s)\|\,ds  \\
&=
\lim_{t\to+\infty}\frac{\|\dot z(t)\|}{pt}
=0.
\end{aligned}
\]
Therefore all terms on the right-hand side of \eqref{eq:qrep} tend to zero, and hence
\begin{equation}\label{eq:S-aux}
   q(t)=S(z(t)-\hat z)\to0.
\end{equation}
Combining \eqref{eq:gradhh}, \eqref{eq:S-aux}, and $F(\hat z)=0$, we obtain
\[
    F(z(t))
    =
    \nabla h(z(t))+Sz(t)
    =
    a(t)+q(t)
    \to0.
\]
Finally, let $\bar z$ be any cluster point of $z(t)$. Then there exists a sequence $t_j\to+\infty$ such that $z(t_j)\to\bar z$. Since $F(z(t))\to0$ and $F$ is continuous, we obtain
\[
    F(\bar z)=0.
\]
Therefore $\bar z\in\cS$. The proof is complete.
\end{proof}
 
\begin{theorem}\label{thm:cont-main}
Let $z(t)=(x(t),y(t))$ be a trajectory of dynamic \eqref{dyn:cont-z}. Then $(x(t),y(t))\to  (x^*,y^*)$ for some $(x^*,y^*)\in\cS$. Assume further that $\alpha>3$ and $\theta\in\left(\frac{1}{\alpha-1},\frac12\right)$. Then the following results hold:
\begin{itemize}
	\item [(i)] The primal-dual gap and the velocity satisfy
	\[
    \cG_{z^*}(z(t))
    =
    \cL(x(t),y^*)-\cL(x^*,y(t))
    =
    o\left(\frac1{t^2}\right),
    \qquad
    \norm{(\dot x(t),\dot y(t))}
    =
    o\left(\frac1t\right).
    \]
    \item [(ii)] If  \Cref{ass:disc} is satisfied, then the stationarity residual satisfies
    \[
    \|F(z(t))\|=\| \nabla h(z(t))+Sz(t)\|
    =
    \left\|\begin{pmatrix} \nabla f(x(t))+K^\top y(t)\\ \nabla g(y(t))-Kx(t) \end{pmatrix}\right\|
    =
    o\left(\frac1t\right).
    \]
\end{itemize}
\end{theorem}

\begin{proof}
By \Cref{thm:cont-est}, the trajectory $z(t)$ is bounded. Hence it has at least one cluster point. By \Cref{lem:cont-cluster}, every cluster point of $z(t)$ belongs to $\cS$.

We now prove that the cluster point is unique. Suppose that $z_a,z_b\in\cS$ are two cluster points. Then there exist sequences $t_j,s_j\to+\infty$ such that
\[
    z(t_j)\to z_a,
    \qquad
    z(s_j)\to z_b.
\]
By \Cref{prop:str}, the primal-dual gaps generated by $z_a$ and $z_b$ coincide:
\[
    \cG_{z_a}(z)=\cG_{z_b}(z),
    \qquad \forall z\in\R^{n}\times\R^m.
\]
Thus the gap parts in $\cE_{z_b}(t)$ and $\cE_{z_a}(t)$ cancel. Using the definition of $\cE_{z^*}$ in \eqref{def:E-cont}, we obtain
\begin{equation}\label{eq:diffE}
\begin{aligned}
\cE_{z_b}(t)-\cE_{z_a}(t)
&=
\frac12\norm{z(t)-z_b+\theta t\dot z(t)}^2
-\frac12\norm{z(t)-z_a+\theta t\dot z(t)}^2  \\
&\quad
+\frac{\xi}{2}\norm{z(t)-z_b}^2
-\frac{\xi}{2}\norm{z(t)-z_a}^2  \\
&=
\frac{\theta(\alpha-1)}{2}
\left(\norm{z(t)-z_b}^2-\norm{z(t)-z_a}^2\right)
+\theta t \ip{z_a-z_b}{\dot z(t)}.
\end{aligned}
\end{equation}
Define
\[
    H(t):=
    \frac12\left(
    \norm{z(t)-z_b}^2-\norm{z(t)-z_a}^2
    \right).
\]
Then $\dot H(t)=\ip{z_a-z_b}{\dot z(t)}$. Since both $\cE_{z_a}$ and $\cE_{z_b}$ are nonincreasing and bounded from below, they have finite limits. Hence, by \eqref{eq:diffE},
\[
    R(t):=\frac{1}{\theta}\bigl(\cE_{z_b}(t)-\cE_{z_a}(t)\bigr)
    =
    t\dot H(t)+(\alpha-1)H(t)
\]
has a finite limit. Multiplying both sides by $t^{\alpha-2}$ and integrating over $[t_0,t]$ gives
\[
   H(t)
    =
    \left(\frac{t_0}{t}\right)^{\alpha-1}H(t_0)
    +
    t^{-(\alpha-1)}
    \int_{t_0}^{t}s^{\alpha-2}R(s)\,ds .
\]
Since $R(t)$ has a finite limit and $\alpha\ge3$, L'H\^opital's rule yields
\[
\lim_{t\to+\infty} H(t)
=
\lim_{t\to+\infty}
t^{-(\alpha-1)}
\int_{t_0}^{t}s^{\alpha-2}R(s)\,ds
=
\frac{\lim_{t\to+\infty} R(t)}{\alpha-1}.
\]
Thus $H(t)$ has a finite limit.

On the other hand, using $z(t_j)\to z_a$ and $z(s_j)\to z_b$, we have
\[
    H(t_j)\to \frac12\norm{z_a-z_b}^2,
    \qquad
    H(s_j)\to -\frac12\norm{z_a-z_b}^2.
\]
Since $H(t)$ has a finite limit, these two limits must be equal. Hence
$
    \frac12\norm{z_a-z_b}^2
    =
    -\frac12\norm{z_a-z_b}^2,
$
and therefore $z_a=z_b$. Thus the bounded trajectory has a unique cluster point in $\cS$. Denote it by
$
    z^*=(x^*,y^*)\in\cS.
$
Consequently, $z(t)=(x(t),y(t))\to z^*=(x^*,y^*)$ as $t\to+\infty$.

We next prove (i). Assume that
$
    \alpha>3,\
    \theta\in(\frac{1}{\alpha-1},\frac12).
$
Set $u(t):=z(t)-z^*$.
Since $z(t)\to z^*$, we have $u(t)\to0$. Expanding the energy function $\cE_{z^*}(t)$ gives
\[
\begin{aligned}
\cE_{z^*}(t)
&=
\theta^2t^2\cG_{z^*}(z(t))
+\frac12\norm{u(t)+\theta t\dot z(t)}^2
+\frac{\xi}{2}\norm{u(t)}^2\\
&=
\theta^2t^2\cG_{z^*}(z(t))
+\frac{\theta^2t^2}{2}\norm{\dot z(t)}^2
+\theta t\ip{u(t)}{\dot z(t)}
+\frac{1+\xi}{2}\norm{u(t)}^2.
\end{aligned}
\]
By \Cref{thm:cont-est}, $t\norm{\dot z(t)}$ is bounded. Since $u(t)\to0$, the last two terms tend to zero. Also, $\cE_{z^*}(t)$ has a finite limit because it is nonincreasing and bounded from below. Therefore there exists $\ell\ge0$ such that
\begin{equation}\label{eq:ocon}
\lim_{t\to+\infty}
\left(
t^2\cG_{z^*}(z(t))
+\frac{t^2}{2}\norm{\dot z(t)}^2
\right)=\frac{1}{\theta^2}\lim_{t\to+\infty}\cE_{z^*}(t)
=\ell\ge 0 .
\end{equation}
We show that $\ell=0$. Suppose, to the contrary, that $\ell>0$. Then there exists $T\ge t_0$ such that, for all $t\ge T$,
$
    t^2\cG_{z^*}(z(t))
    +\frac{t^2}{2}\norm{\dot z(t)}^2
    \ge \frac{\ell}{2}.
$
Hence
\[
\begin{aligned}
\int_T^{+\infty}
\left(
t\cG_{z^*}(z(t))
+\frac{t}{2}\norm{\dot z(t)}^2
\right)\,dt
&\ge
\int_T^{+\infty}\frac{\ell}{2t}\,dt
=+\infty.
\end{aligned}
\]
This contradicts the  estimates in \Cref{thm:cont-est}. Thus $\ell=0$, and \eqref{eq:ocon} implies (i).

We finally prove (ii). Suppose that \Cref{ass:disc} is satisfied. By (i), \eqref{eq:gap}, and \Cref{lem:D-smooth}(ii) with $u=z(t)$ and $v=z^*$, we have
$   \|\nabla h(z(t))-\nabla h(z^*)\|
   =
   o\left(1/t\right).
$
Equivalently,
\begin{equation}\label{eq:thto0}
 	\lim_{t\to+\infty}t\|\nabla h(z(t))-\nabla h(z^*)\| =0.
\end{equation}
Define $q(t)$ and $a(t)$ as in \eqref{def:qa} with $\hat z=z^*$. Then
$
    \lim_{t\to+\infty}t\|a(t)\|=0.
$
We now prove that
$
    \lim_{t\to+\infty}t\|q(t)\|=0.
$
Multiplying \eqref{eq:qrep} by $t$, we obtain
\begin{align}\label{eq:tqto0}
tq(t)
=&\frac{t_0^p}{t^{p-1}} q(t_0)
  -pt^{1-p}\int_{t_0}^{t}s^{p-1}a(s)\,ds \notag\\
&-pt^{1-p}\int_{t_0}^{t}s^{p-1}\ddot z(s)\,ds
  -p\alpha t^{1-p}\int_{t_0}^{t}s^{p-2}\dot z(s)\,ds,
\end{align}
where $p=\frac1\theta\ge2$. We first observe that the first term on the right-hand side of \eqref{eq:tqto0} tends to zero. For the second term, L'H\^opital's rule, together with $\lim_{t\to+\infty}t\|a(t)\|=0$, shows that it also tends to zero. By the same argument as in \Cref{lem:cont-cluster}, and using $\|\dot z(t)\|=o(1/t)$, we obtain that the last two terms also vanish. Therefore, \eqref{eq:tqto0} gives
\[
\lim_{t\to+\infty}t\|q(t)\|
=
\lim_{t\to+\infty}t\|S(z(t)-z^*)\|
=0.
\]
Combining this with \eqref{eq:thto0} and \eqref{eq:eqvF}, we obtain
\[
\lim_{t\to+\infty}t\|\nabla h(z(t))+Sz(t)\|
\le
\lim_{t\to+\infty}t\|\nabla h(z(t))-\nabla h(z^*)\|
+
\lim_{t\to+\infty}t\|S(z(t)-z^*)\|
=0.
\]
This proves (ii).
\end{proof}

\begin{remark}
\Cref{thm:cont-main} extends the convergence results for Nesterov-type dynamics from convex minimization \cite{AttouchMp,MayTJM,Jang,HeTc,BotJDE} to the bilinear saddle point problem \eqref{prob:saddle}. Compared with existing works on accelerated saddle dynamics with bilinear coupling \cite{ZengIFAC,HeAMO,DingCOAP}, which only establish accelerated $\bigO(1/t^{2})$ convergence rates for the primal-dual gap and boundedness of trajectories, \Cref{thm:cont-main} further proves convergence of the whole trajectory and improved little-$o$ rates for both the primal-dual gap and the velocity. Under the additional smoothness assumption, it also yields a little-$o$ estimate for the stationarity residual $\|F(z(t))\|$. Thus \Cref{thm:cont-main}  provides a continuous-time analogue of Nesterov asymptotic convergence in the bilinear saddle setting.
\end{remark}

  \section{Fast primal-dual algorithm}\label{sec:discrete}

In this section, we derive and analyze a fast primal-dual algorithm for the saddle point problem \eqref{prob:saddle} under  \Cref{ass:disc}. The algorithm is obtained from the continuous dynamic \eqref{dyn:cont-z} through a structure-preserving discretization.  

 \subsection{Finite-difference discretization and accelerated primal-dual scheme}
 We first explain how the discrete scheme is motivated by a direct finite-difference discretization of the continuous dynamic \eqref{dyn:cont-z}. Choose a stepsize $\sqrt h>0$ and set
$
    \tau_k=(k+\alpha-1)\sqrt h$ and $
    z_k\approx z(\tau_k).
$
At the node $\tau_k$, we use the standard second-order difference for $\ddot z(\tau_k)$ and the backward difference for $\dot z(\tau_k)$. The gradient term is evaluated at an extrapolated point, while the corrected skew-symmetric term is treated semi-implicitly. This gives
\begin{equation}\label{eq:fdbas}
\frac{z_{k+1}-2z_k+z_{k-1}}{h}
+\frac{\alpha}{\tau_k}\frac{z_k-z_{k-1}}{\sqrt h}
+\nabla h(\bar z_k)
+S\left(z_{k+1}
+\theta k(z_{k+1}-z_k)\right)=0.
\end{equation}       
 Here the extrapolated point is chosen according to Nesterov's accelerated scheme:
\begin{equation}\label{def:barfd}
    \bar z_k=(\bar x_k,\bar y_k)
    :=
    z_k+\frac{k-1}{k+\alpha-1}(z_k-z_{k-1}).
\end{equation}
Since $\tau_k=(k+\alpha-1)\sqrt h$, we have
\[
\begin{aligned}
\frac{z_{k+1}-2z_k+z_{k-1}}{h}
+\frac{\alpha}{\tau_k}\frac{z_k-z_{k-1}}{\sqrt h}
&=
\frac1h
\left[
z_{k+1}-z_k
-\frac{k-1}{k+\alpha-1}(z_k-z_{k-1})
\right]\\
&=
\frac1h(z_{k+1}-\bar z_k).
\end{aligned}
\]
Therefore, multiplying \eqref{eq:fdbas} by $h$, we obtain
\[
    z_{k+1}-\bar z_k
    +h\nabla h(\bar z_k)
    +hS\left(z_{k+1}
    +\theta k(z_{k+1}-z_k)\right)=0.
\]
Writing $z_k=(x_k,y_k)$ and using
$S=
  \begin{pmatrix}
        0&K^\top\\
        -K&0
    \end{pmatrix},
$
we obtain the finite-difference primal-dual scheme
\begin{equation}\label{alg:fd-xy}
\begin{cases}
	x_{k+1}-\bar x_k+h\left(\nabla f(\bar x_k)+K^\top\bigl(y_{k+1}+\theta k(y_{k+1}-y_k)\bigr)\right)=0,\\[0.8ex]
	y_{k+1}-\bar y_k+h\left(\nabla g(\bar y_k)-K\bigl(x_{k+1}+\theta k(x_{k+1}-x_k)\bigr)\right)=0.
\end{cases}
\end{equation}

The scheme \eqref{alg:fd-xy} uses a single stepsize $h$ for both variables. Since $f$ and $g$ may have different smoothness constants, we allow different primal and dual stepsizes $r$ and $s$, respectively. In addition, the coefficient $(k-1)/(k+\alpha-1)$ in \eqref{def:barfd} is only one particular Nesterov-type coefficient. We therefore replace it by a more general accelerated sequence. Let $\{t_k\}$ satisfy
\begin{equation}\label{ass:tk}
    t_1=1,\qquad
    t_k\to+\infty,\qquad
    t_{k+1}^2\le t_k^2+\rho t_{k+1},
    \qquad
    \rho\in(0,1].
\end{equation}
Under this parameter setting, we have 
$t_{k+1}-t_k=\dfrac{t_{k+1}^2-t_k^2}{t_{k+1}+t_k}
\le \dfrac{\rho t_{k+1}}{t_{k+1}+t_k}\le \rho\le1$.
Hence $t_k\le t_1+\sum_{i=1}^{k-1}(t_{i+1}-t_i)\le 1+\rho(k-1)\le k$.
Therefore $1/t_k\ge 1/k$, and consequently
\begin{equation}\label{eq:tkinf}
	\sum_{k=1}^{+\infty}\frac{1}{t_k}=+\infty.
\end{equation}

We define
\begin{equation}\label{def:bargen}
    \bar z_k=(\bar x_k,\bar y_k)
    :=
    z_k+\frac{t_k-1}{t_{k+1}}(z_k-z_{k-1}),
\end{equation}
that is,
\[
    \bar x_k=x_k+\frac{t_k-1}{t_{k+1}}(x_k-x_{k-1}),
    \qquad
    \bar y_k=y_k+\frac{t_k-1}{t_{k+1}}(y_k-y_{k-1}).
\]
The correction coefficient $\theta k$ in \eqref{alg:fd-xy} is also replaced by the parameter
$
    a_k:=\frac{t_{k+1}-\eta}{\eta}$ with
   $ \eta\in[\rho,1]$.
With these replacements, the finite-difference structure leads to the following accelerated primal-dual scheme:
\begin{equation}\label{alg:coupled-xy}
\begin{cases}
x_{k+1}-\bar x_k
+r\left(\nabla f(\bar x_k)
+K^\top\bigl(y_{k+1}+a_k(y_{k+1}-y_k)\bigr)\right)=0,\\[0.8ex]
y_{k+1}-\bar y_k
+s\left(\nabla g(\bar y_k)
-K\bigl(x_{k+1}+a_k(x_{k+1}-x_k)\bigr)\right)=0.
\end{cases}
\end{equation}

\begin{remark}
The generalized extrapolation coefficient in \eqref{def:bargen} contains the finite-difference coefficient in \eqref{def:barfd} as a special case. Indeed, if $t_k=\frac{k+\alpha-2}{\alpha-1}$, then
$\frac{t_k-1}{t_{k+1}}
    =
    \frac{k-1}{k+\alpha-1}$.
In this case, one can verify that \eqref{ass:tk} holds with $\rho=\frac{2}{\alpha-1}$. Taking $\eta=1$ gives $\theta=\frac{1}{\alpha-1}$
Thus, the correction term reduces to that in \eqref{alg:fd-xy} with $\theta=\frac{1}{\alpha-1}$. In the convergence analysis, we use the more flexible choice $a_k=(t_{k+1}-\eta)/\eta$, which preserves the same accelerated correction structure.
\end{remark}

To implement \eqref{alg:coupled-xy}, we rewrite it into an equivalent sequential form. For simplicity, denote
$
    c_k:=1+a_k.
$
From the second equation of \eqref{alg:coupled-xy}, we have
\[
    y_{k+1}
     = 
    \bar y_k-s\nabla g(\bar y_k)
    +s K(c_kx_{k+1}-a_kx_k).
\]
Substituting this expression into the first equation of \eqref{alg:coupled-xy}, we obtain
\[
\begin{aligned}
0
=&r^{-1}(x_{k+1}-\bar x_k)
+\nabla f(\bar x_k) \\
&+K^\top\left(c_k(\bar y_k-s\nabla g(\bar y_k))-a_ky_k\right)
+s c_kK^\top K(c_kx_{k+1}-a_kx_k).
\end{aligned}
\]
Therefore, $x_{k+1}$ is equivalently characterized as the unique minimizer of the following strongly convex quadratic subproblem:
\begin{equation}\label{eq:xsub}
\begin{aligned} 
 x_{k+1}
=
\arg\min_{x\in\R^n}
\Bigg\{
&\frac{1}{2r}\|x-\bar x_k\|^2
+\ip{\nabla f(\bar x_k)}{x}
+\ip{Kx}{c_k(\bar y_k-s\nabla g(\bar y_k))-a_ky_k}  \\
&\qquad
+\frac{s}{2}\|c_kKx-a_kKx_k\|^2
\Bigg\}.
\end{aligned}
\end{equation}
After $x_{k+1}$ is computed, $y_{k+1}$ is updated by the second equation of \eqref{alg:coupled-xy}. This gives the following iterative algorithm.

 \begin{algorithm}[H]
\caption{Fast primal-dual algorithm for \eqref{prob:saddle}}\label{alg:fast-pd}
\begin{algorithmic}
\State Choose $x_0=x_1\in\R^n$, $y_0=y_1\in\R^m$, stepsizes $0<r\le \frac1{L_f}$ and $0<s\le \frac1{L_g}$, a nondecreasing sequence $\{t_k\}$ satisfying \eqref{ass:tk}, and $a_k=\frac{t_{k+1}-\eta}{\eta}$ with $\eta\in[\rho,1]$.
\For{$k=1,2,\ldots$}
\State Compute the extrapolated points
\[
    \bar x_k=x_k+\frac{t_k-1}{t_{k+1}}(x_k-x_{k-1}),
    \qquad
    \bar y_k=y_k+\frac{t_k-1}{t_{k+1}}(y_k-y_{k-1}).
\]
\State Set $c_k=1+a_k$ and compute $x_{k+1}$ by solving
\[
\begin{aligned}
x_{k+1}
=
\arg\min_{x\in\R^n}
\Bigg\{
&\frac{1}{2r}\|x-\bar x_k\|^2
+\ip{\nabla f(\bar x_k)}{x}
+\ip{Kx}{c_k(\bar y_k-s\nabla g(\bar y_k))-a_ky_k}  \\
&\qquad
+\frac{s}{2}\|c_kKx-a_kKx_k\|^2
\Bigg\}.
\end{aligned}
\]
\State Update dual varible 
\[
    y_{k+1}
    =
 \bar y_k-s\nabla g(\bar y_k)+sK(c_kx_{k+1}-a_kx_k).
\]
\EndFor
\end{algorithmic}
\end{algorithm}

 Note that Algorithm \ref{alg:fast-pd} is equivalent to the coupled system \eqref{alg:coupled-xy}. Indeed, the dual update in the algorithm is obtained directly from the second equation of \eqref{alg:coupled-xy}. Substituting this expression of $y_{k+1}$ into the first equation gives exactly the optimality condition of the subproblem \eqref{eq:xsub}. Conversely, the optimality condition of \eqref{eq:xsub}, together with the dual update, recovers the two equations in \eqref{alg:coupled-xy}. Hence both formulations generate the same sequence. In the sequel, a sequence generated by \Cref{alg:fast-pd} is understood equivalently as a sequence satisfying \eqref{alg:coupled-xy}.

\subsection{Convergence analysis}

We now prove convergence properties of \Cref{alg:fast-pd}. In this part, we use the compact notation
\[
    z_k=(x_k,y_k),\quad \bar z_k=(\bar x_k,\bar y_k),
    \quad
    D=
    \begin{pmatrix}
        r^{-1}I_n&0\\
        0&s^{-1}I_m
    \end{pmatrix}.
\]
The coupled form of \eqref{alg:coupled-xy} can be written as
\begin{equation}\label{alg:coupZ}
    D(z_{k+1}-\bar z_k)+\nabla h(\bar z_k)
    +S\bigl(z_{k+1}+a_k(z_{k+1}-z_k)\bigr)=0.
\end{equation}
We shall use \eqref{alg:coupZ} in the convergence analysis of \Cref{alg:fast-pd}.

For any $z^*=(x^*,y^*)\in\cS$, define the   sequence
\begin{equation}\label{def:E-disc}
    \cE_k^{z^*}:=
    t_k^2\cG_{z^*}(z_k)+B_k^{z^*},
\end{equation}
where
\begin{equation}\label{def:B-disc}
    B_k^{z^*}:=
    \frac12\|w_k^{z^*}\|_D^2
    +\frac{\eta(1-\eta)}2\|z_k-z^*\|_D^2
\end{equation}
with
\begin{equation*}\label{def:w-disc}
    w_k^{z^*}:=\eta(z_k-z^*)+(t_k-1)(z_k-z_{k-1}).
\end{equation*}

The following lemma shows that $\cE_k^{z^*}$ is a discrete energy sequence and gives the summability estimates needed for the improved convergence rates.

\begin{lemma}\label{lem:disc}
Let $\{z_k\}=\{(x_k,y_k)\}$ be the sequence generated by \Cref{alg:fast-pd}. Then, for every $z^*\in\cS$, the sequence $\{\cE_k^{z^*}\}$ is  nonnegative and nonincreasing, consequently, $\lim_{k\to+\infty}\cE_k^{z^*}$ exists. Moreover, if $\rho<\eta<1$, then
\[
    \sum_{k=1}^{+\infty}t_{k+1}\cG_{z^*}(z_k)<+\infty,
    \qquad
    \sum_{k=1}^{+\infty}t_{k+1}\|z_{k+1}-z_k\|_D^2<+\infty.
\]
\end{lemma}

\begin{proof}
From the definitions of $\bar z_k$ and $w_k^{z^*}$, we have
\begin{eqnarray}\label{eq:wdiff}
	  w_{k+1}^{z^*}-w_k^{z^*}
	  &=&(t_{k+1}+\eta-1)(z_{k+1}-z_k)-(t_k-1)(z_k-z_{k-1})\notag\\
	  &=&t_{k+1}\left(z_{k+1}-z_k-\frac{t_k-1}{t_{k+1}}(z_k-z_{k-1})\right)
	  -(1-\eta)(z_{k+1}-z_k)\notag\\
	  &=& t_{k+1}(z_{k+1}-\bar z_k)
    -(1-\eta)(z_{k+1}-z_k).
\end{eqnarray}
Hence
\begin{align*} 
 \frac12\|w_{k+1}^{z^*}-w_k^{z^*}\|_D^2
={}&
\frac{t_{k+1}^2}{2}\|z_{k+1}-\bar z_k\|_D^2
+\frac{(1-\eta)^2}{2}\|z_{k+1}-z_k\|_D^2
\\
&-
(1-\eta)t_{k+1}
\left\langle
D(z_{k+1}-z_k),z_{k+1}-\bar z_k
\right\rangle .
\end{align*}
Using the identity
$
    \frac12\|a\|_D^2-\frac12\|b\|_D^2
    =
    \ip{Da}{a-b}-\frac12\|a-b\|_D^2,
$
and applying it to \eqref{def:B-disc} and \eqref{eq:wdiff}, we obtain
\begin{equation}\label{eq:Bdiff} 
  \begin{aligned}
 B_{k+1}^{z^*}- B_k^{z^*}
&=\ip{Dw_{k+1}^{z^*}}{w_{k+1}^{z^*}-w_k^{z^*}}
-\frac12\|w_{k+1}^{z^*}-w_k^{z^*}\|_D^2\\
&\quad
+ \eta(1-\eta)\ip{D(z_{k+1}-z^*)}{z_{k+1}-z_k}
-\frac{\eta(1-\eta)}{2}\|z_{k+1}-z_k\|_D^2 \\
&=
\eta t_{k+1}\ip{D(z_{k+1}-z^*)}{z_{k+1}-\bar z_k}\\
&\quad
+t_{k+1}(t_{k+1}-\eta)
\ip{D(z_{k+1}-z_k)}{z_{k+1}-\bar z_k}\\
&\quad
-(1-\eta)\left(t_{k+1}-\frac{1}{2}\right)
\|z_{k+1}-z_k\|_D^2
-\frac{t_{k+1}^2}{2}\|z_{k+1}-\bar z_k\|_D^2 .
\end{aligned}
\end{equation}
Since $z^*\in\cS$, we have $\nabla h(z^*)+Sz^*=0$. Subtracting this identity from \eqref{alg:coupZ} yields
\[
    D(z_{k+1}-\bar z_k)
    +\nabla h(\bar z_k)-\nabla h(z^*)
    +S\bigl(z_{k+1}-z^*+a_k(z_{k+1}-z_k)\bigr)=0.
\]
Since $\eta a_k=t_{k+1}-\eta$, the first two inner products in \eqref{eq:Bdiff} can be combined as
\[
\begin{aligned}
&\eta t_{k+1}\ip{D(z_{k+1}-z^*)}{z_{k+1}-\bar z_k}
+t_{k+1}(t_{k+1}-\eta)
\ip{D(z_{k+1}-z_k)}{z_{k+1}-\bar z_k}\\
&=
\eta t_{k+1}
\ip{z_{k+1}-z^*+\alpha_k(z_{k+1}-z_k)}
     {D(z_{k+1}-\bar z_k)}\\
&=
-\eta t_{k+1}
\ip{z_{k+1}-z^*+\alpha_k(z_{k+1}-z_k)}
     {\nabla h(\bar z_k)-\nabla h(z^*)},
\end{aligned}
\]
where the term involving $S$ vanishes by skew-symmetry:
$
    \ip{S\xi}{\xi}=0$ with
    $\xi=z_{k+1}-z^*+a_k(z_{k+1}-z_k).
$
Therefore,
\begin{equation}\label{eq:Bbefore}
\begin{aligned}
B_{k+1}^{z^*}-B_k^{z^*}
&=
-\eta t_{k+1}
\ip{z_{k+1}-z^*}{\nabla h(\bar z_k)-\nabla h(z^*)}\\
&\quad
-t_{k+1}(t_{k+1}-\eta)
\ip{z_{k+1}-z_k}{\nabla h(\bar z_k)-\nabla h(z^*)}\\
&\quad
-(1-\eta)\left(t_{k+1}-\frac12\right)
\|z_{k+1}-z_k\|_D^2
-\frac{t_{k+1}^2}{2}\|z_{k+1}-\bar z_k\|_D^2.
\end{aligned}
\end{equation}
By the definition of $\cG_{z^*}$ in \eqref{eq:gap} and \Cref{lem:D-smooth}(i) with $u=z_{k+1}$ and $v=\bar z_k$, for any $z\in\R^n\times\R^m$ we have
\[\begin{aligned}
&\cG_{z^*}(z_{k+1})
=h(z_{k+1})-h(z^*)-\ip{\nabla h(z^*)}{z_{k+1}-z^*}\\
&\quad\le h(\bar z_k)-h(z^*)-\ip{\nabla h(z^*)}{z_{k+1}-z^*}
+\ip{\nabla h(\bar z_k)}{z_{k+1}-\bar z_k}
+\frac12\|z_{k+1}-\bar z_k\|_D^2\\
&\quad\le h(z)-h(z^*)-\ip{\nabla h(z^*)}{z_{k+1}-z^*}
+\ip{\nabla h(\bar z_k)}{z_{k+1}-z}
+\frac12\|z_{k+1}-\bar z_k\|_D^2\\
&\quad=
\cG_{z^*}(z)
+\ip{\nabla h(\bar z_k)-\nabla h(z^*)}{z_{k+1}-z}
+\frac12\|z_{k+1}-\bar z_k\|_D^2.
\end{aligned}
\]
Here the second inequality follows from the convexity of $h$, which gives
\[
    h(\bar z_k)\le h(z)+\ip{\nabla h(\bar z_k)}{\bar z_k-z},
    \qquad \forall z\in\R^n\times\R^m.
\] 
 Taking $z=z^*$ and $z=z_k$, respectively, and using $\cG_{z^*}(z^*)=0$, we obtain
\[
    \ip{z_{k+1}-z^*}{\nabla h(\bar z_k)-\nabla h(z^*)}
    \ge
    \cG_{z^*}(z_{k+1})
    -\frac12\|z_{k+1}-\bar z_k\|_D^2
\]
and
\[
    \ip{z_{k+1}-z_k}{\nabla h(\bar z_k)-\nabla h(z^*)}
    \ge
    \cG_{z^*}(z_{k+1})-\cG_{z^*}(z_k)
    -\frac12\|z_{k+1}-\bar z_k\|_D^2.
\]
Substituting these two inequalities into \eqref{eq:Bbefore} and using
$
    \eta t_{k+1}+t_{k+1}(t_{k+1}-\eta)=t_{k+1}^2,
$ 
we get
\[
\begin{aligned}
B_{k+1}^{z^*}-B_k^{z^*}
&\le
-\eta t_{k+1}\cG_{z^*}(z_{k+1})
-t_{k+1}(t_{k+1}-\eta)
\bigl(\cG_{z^*}(z_{k+1})-\cG_{z^*}(z_k)\bigr)\\
&\quad
-(1-\eta)\left(t_{k+1}-\frac12\right)
\|z_{k+1}-z_k\|_D^2.
\end{aligned}
\]
Therefore, from \eqref{def:E-disc} and \eqref{ass:tk}, we have
\begin{equation}\label{eq_decE}
\begin{aligned}
&\cE_{k+1}^{z^*}-\cE_k^{z^*}
= t_{k+1}^2\cG_{z^*}(z_{k+1})
-t_k^2\cG_{z^*}(z_k)+B_{k+1}^{z^*}-B_k^{z^*}\\
&\quad\le
\bigl(t_{k+1}^2-t_k^2-\eta t_{k+1}\bigr)\cG_{z^*}(z_k)
-(1-\eta)\left(t_{k+1}-\frac12\right)
\|z_{k+1}-z_k\|_D^2\\
&\quad\le
(\rho-\eta) t_{k+1}\cG_{z^*}(z_k)
-(1-\eta)\left(t_{k+1}-\frac12\right)   
\|z_{k+1}-z_k\|_D^2.
\end{aligned}
\end{equation}
Since $\rho\le \eta\le 1$ and $t_{k+1}\ge1$, it follows that
\[
    \cE_{k+1}^{z^*}-\cE_k^{z^*}\le 0,
    \qquad \forall k\ge1.
\]
Since $\cE_k^{z^*}\ge0$, then the sequence $\{\cE_k^{z^*}\}$ is nonnegative and nonincreasing. Consequently, $\lim_{k\to+\infty}\cE_k^{z^*}$ exists.

If $\rho<\eta<1$, summing \eqref{eq_decE} over $k\ge1$ and using the boundedness from below of $\{\cE_k^{z^*}\}$ give
\[
    \sum_{k=1}^{+\infty}t_{k+1}\cG_{z^*}(z_k)<+\infty,
    \qquad
    \sum_{k=1}^{+\infty}t_{k+1}\|z_{k+1}-z_k\|_D^2<+\infty.
\]
This completes the proof.
\end{proof}

   We now use the discrete energy estimate to establish the convergence rate, the convergence of the  generated sequence, and the improved asymptotic rates in the noncritical regime.

 \begin{theorem}\label{thm:disc-main}
Let $\{z_k\}=\{(x_k,y_k)\}$ be the sequence generated by \Cref{alg:fast-pd}. Then the following statements hold:
\begin{enumerate}
\item[(i)] For every $(x^*,y^*)\in\cS$,
\[
\cL(x_k,y^*)-\cL(x^*,y_k)=\bigO\left(\frac1{t_k^2}\right),
\qquad
\|(x_k,y_k)-(x_{k-1},y_{k-1})\|=\bigO\left(\frac1{t_k}\right).
\]

\item[(ii)] There exists $(x^*,y^*)\in\cS$ such that $(x_k,y_k)$ converges to $(x^*,y^*)$ as $k\to+\infty$.

\item[(iii)] If $\rho<\eta<1$, then the following improved rates hold:
\begin{itemize}
	\item[(a)] Improved convergence rates for the primal-dual gap and the velocity:
 \[
    \cL(x_k,y^*)-\cL(x^*,y_k)
    =
    o\left(\frac1{t_k^2}\right),
    \qquad
  \|(x_k,y_k)-(x_{k-1},y_{k-1})\|
    =
    o\left(\frac1{t_k}\right).
\]
\item[(b)] Improved convergence rate for the stationarity residual:
\[
\|F(z_k)\|
=
\|\nabla h(z_k)+Sz_k\|
=
\left\|\begin{pmatrix}
        \nabla f(x_k)+K^\top y_k\\
        \nabla g(y_k)-Kx_k
    \end{pmatrix}\right\|
    =
    o\left(\frac1{t_k}\right).
\]
\end{itemize}
\end{enumerate}
\end{theorem}

\begin{proof}
  
 Let $z^*\in\cS$ be fixed. By \Cref{lem:disc}, $\{\cE_k^{z^*}\}$ is nonincreasing and nonnegative. Hence
$
    \cE_k^{z^*}\le \cE_1^{z^*}.
$
Together with the definition of $\cE_k^{z^*}$, this gives
\[
    t_k^2\cG_{z^*}(z_k)
    =
    t_k^2(\cL(x_k,y^*)-\cL(x^*,y_k))
    \le \cE_1^{z^*},
\]
and
\[
\|w_k^{z^*}\|_D
=
\|\eta(z_k-z^*)+(t_k-1)(z_k-z_{k-1})\|_D
\le \sqrt{2\cE_1^{z^*}}.
\]
By \Cref{le_wbound} and the positive definiteness of $D$, we obtain (i). In particular, the sequence $\{z_k\}=\{(x_k,y_k)\}$ is bounded.

 We next prove (ii). We first show that every cluster point is a saddle point. Fix $\hat z=(\hat x,\hat y)\in\cS$. From (i), we have
\[
    \cG_{\hat z}(z_k)=\cL(x_k,\hat y)-\cL(\hat x,y_k)\to0.
\]
Let $\bar z$ be any cluster point of $\{z_k\}$, and let $\{z_{k_j}\}$ be a subsequence such that $z_{k_j}\to\bar z$. Then, by an argument similar to that in \Cref{lem:cont-cluster}, we have $\nabla h(\bar z)=\nabla h(\hat z)$.

Since $\{z_k\}$ is bounded and $\nabla h$ is continuous, the preceding argument shows that all cluster points of $\{\nabla h(z_k)\}$ are equal to $\nabla h(\hat z)$. Hence
\begin{equation}\label{eq_hzhz}
	    \nabla h(z_k)\to\nabla h(\hat z).
\end{equation}

Moreover, by the definition of $\bar z_k$ and (i), we have
\begin{equation}\label{eq_diff_z}
	\|\bar z_k-z_k\|
	\le \|z_k-z_{k-1}\|\to 0.
\end{equation}
Using  \Cref{ass:disc}, \eqref{eq_hzhz}, and \eqref{eq_diff_z}, we obtain
\begin{equation}\label{eq:gradbar}
\begin{aligned}
	  \|\nabla h(\bar z_k)-\nabla h(\hat z)\|
	  &\leq \|\nabla h(z_k)-\nabla h(\bar z_k)\|
	  + \|\nabla h(z_k)-\nabla h(\hat z)\|\\
	  &\le \max\{L_f,L_g\}\|\bar z_k-z_k\|
	  +\|\nabla h(z_k)-\nabla h(\hat z)\|\\
	  &\to 0.
\end{aligned}
\end{equation}
Thus
\begin{equation}\label{eq_nh_to0}
	\nabla h(\bar z_k)\to\nabla h(\hat z).
\end{equation}

Set
\[
    q_k:=S(z_k-\hat z).
\]
Since $\hat z\in\cS$, we have
$
    F(\hat z)=\nabla h(\hat z)+S\hat z=0.
$
Subtracting this identity from \eqref{alg:coupZ} and using $a_k=\frac{t_{k+1}-\eta}{\eta}$, we obtain
\[
\begin{aligned}
0
&=
D(z_{k+1}-\bar z_k)
+\nabla h(\bar z_k)-\nabla h(\hat z)
+S\bigl(z_{k+1}-\hat z+a_k(z_{k+1}-z_k)\bigr)\\
&=
D(z_{k+1}-\bar z_k)
+\nabla h(\bar z_k)-\nabla h(\hat z)
+\frac{t_{k+1}}{\eta}q_{k+1}
-\frac{t_{k+1}-\eta}{\eta}q_k .
\end{aligned}
\]
It follows that
\begin{equation}\label{eq:qrec}
    q_{k+1}
    =
    \left(1-\frac{\eta}{t_{k+1}}\right)q_k
    -\frac{\eta}{t_{k+1}}
    \left[
    D(z_{k+1}-\bar z_k)+\nabla h(\bar z_k)-\nabla h(\hat z)
    \right].
\end{equation} 
  
 By (i) and \eqref{eq_diff_z}, we have
\[
    z_{k+1}-\bar z_k
    =
    (z_{k+1}-z_k)-(\bar z_k-z_k)\to0.
\]
Therefore
$
    D(z_{k+1}-\bar z_k)\to0.
$
Together with \eqref{eq_nh_to0}, this gives
\[
	D(z_{k+1}-\bar z_k)+\nabla h(\bar z_k)-\nabla h(\hat z)\to0.
\]
Since $\eta\in[\rho,1]$, we have $\eta/t_{k+1}\in(0,1]$. By \eqref{eq:tkinf},
$
    \sum_{k=1}^{+\infty}\frac{\eta}{t_{k+1}}=+\infty.
$
Hence \Cref{le_bertsekas} applied to \eqref{eq:qrec} gives
$
    q_k=S(z_k-\hat z)\to0.
$
Combining this with \eqref{eq_hzhz} and $F(\hat z)=\nabla h(\hat z)+S\hat z=0$, we obtain
\[
F(z_k)
=
\nabla h(z_k)-\nabla h(\hat z)+S(z_k-\hat z)
\to 0.\]
Thus, if $\bar z$ is any cluster point of $\{z_k\}$ and $z_{k_j}\to\bar z$, then the continuity of $F$ gives
\[
    F(\bar z)=\lim_{j\to+\infty}F(z_{k_j})=0.
\]
Therefore $\bar z\in\cS$. Hence every cluster point of $\{z_k\}$ belongs to $\cS$.

We now prove that the cluster point is unique. Let $z_a,z_b\in\cS$ be two cluster points. By \Cref{prop:str},
\[
    \cG_{z_a}(z)=\cG_{z_b}(z),
    \qquad \forall z\in\R^n\times\R^m.
\]
Thus the primal-dual gap terms in $\cE_k^{z_a}$ and $\cE_k^{z_b}$ are identical. Since both energy sequences are nonincreasing and bounded from below, the sequence
\begin{equation}\label{eq_RR}
	R_k:=\frac1\eta\left(\cE_k^{z_a}-\cE_k^{z_b}\right)
	  =
	  \frac{1}{2}(\|z_k-z_a\|_D^2-\|z_k-z_b\|_D^2)
	  +(t_k-1)\ip{D(z_b-z_a)}{z_k-z_{k-1}}
\end{equation}
has a finite limit, denoted by $R^{\infty}$.

Define
\[
    H_k:=
    \frac12\|z_k-z_a\|_D^2
    -
    \frac12\|z_k-z_b\|_D^2
    =
    \ip{D(z_b-z_a)}{z_k}
    +\frac{1}{2}(\|z_a\|_D^2-\|z_b\|_D^2).
\]
Then, from \eqref{eq_RR}, we have
$
    R_k=H_k+(t_k-1)(H_k-H_{k-1}).
$
Since $t_k\ge1$, it follows that
\[
|H_k-R^{\infty}|
\le
\frac{t_k-1}{t_k}|H_{k-1}-R^{\infty}|
+\frac{1}{t_k}|R_k-R^{\infty}|.
\]
By \eqref{eq:tkinf} and \Cref{le_bertsekas}, we conclude that
$
    \lim_{k\to+\infty}H_k=R^{\infty}.
$
 Now take subsequences $z_{k_j}\to z_a$ and $z_{\ell_j}\to z_b$. Then
\[
    H_{k_j}\to -\frac12\|z_a-z_b\|_D^2,
    \qquad
    H_{\ell_j}\to \frac12\|z_a-z_b\|_D^2.
\]
Since $H_k$ has a unique limit, these two limits must be equal. Therefore $z_a=z_b$. Thus the bounded sequence $\{z_k\}$ has a unique cluster point, denoted by $z^*=(x^*,y^*)\in\cS$. Consequently,
$
    z_k=(x_k,y_k)\to z^*=(x^*,y^*)
$
as $k\to+\infty$. This proves (ii).

We finally prove (iii). Let $z_k\to z^*$. Since $\rho<\eta<1$, \Cref{lem:disc} implies that $\cE_k^{z^*}$ has a finite limit. Define
\begin{equation}\label{eq_difA}
	   A_k:=
    t_k^2\cG_{z^*}(z_k)
    +
    \frac{(t_k-1)^2}{2}\|z_k-z_{k-1}\|_D^2.
\end{equation}
Expanding $\cE_k^{z^*}$ gives
\begin{eqnarray}\label{eq_limA}
	\lim_{k\to+\infty}\cE_k^{z^*}
	&=&
	\lim_{k\to+\infty}
	\left(
	A_k+\frac{\eta}{2}\|z_k-z^*\|_D^2
	+\eta(t_k-1)\ip{D(z_k-z^*)}{z_k-z_{k-1}}
	\right)\nonumber\\
	&=&
	\lim_{k\to+\infty} A_k,
\end{eqnarray}
where the last equality follows from (i) and (ii), since $z_k\to z^*$ and $(t_k-1)\|z_k-z_{k-1}\|_D$ is bounded. Furthermore, by \Cref{lem:disc} and the monotonicity of $\{t_k\}$, we have
\begin{equation}\label{eq_AK}
	 \sum_{k=1}^{+\infty}\frac{A_k}{t_k}<+\infty .
\end{equation}
From \eqref{eq_difA} and \eqref{eq_limA}, there exists $\ell\ge0$ such that
$
    A_k\to\ell.
$
If $\ell>0$, then there exists $K\ge1$ such that $A_k\ge\ell/2$ for all $k\ge K$. Hence
\[
    \sum_{k=K}^{+\infty}\frac{A_k}{t_k}
    \ge
    \frac{\ell}{2} \sum_{k=K}^{+\infty}\frac{1}{t_k}
    =
    +\infty,
\]
which contradicts \eqref{eq_AK}. Therefore $\ell=0$. By the definition of $A_k$ and the positive definiteness of $D$, we obtain
\[
	  t_k^2\cG_{z^*}(z_k)\to0,
    \qquad
    t_k\|z_k-z_{k-1}\|\to0.
\]
This proves (a).

It remains to prove the stationarity residual estimate. By \Cref{lem:D-smooth}(ii) with $u=z_k$ and $v=z^*$, and by (a), we have
\begin{equation}\label{eq_onzk}
	    \|\nabla h(z_k)-\nabla h(z^*)\|^2
    \le
    2\max\{L_f,L_g\}\cG_{z^*}(z_k)
    =
    o\left(\frac1{t_k^2}\right).
\end{equation}
Thus
$
    \|\nabla h(z_k)-\nabla h(z^*)\|=o(1/t_k).
$
Together with \eqref{eq_diff_z}, \eqref{eq:gradbar}, and (a), this gives
\begin{equation}\label{eq:grado}
	\|\nabla h(\bar z_k)-\nabla h(z^*)\|
	\le
	\max\{L_f,L_g\}\|\bar z_k-z_k\|
	+\|\nabla h(z_k)-\nabla h(z^*)\|
	=
	o\left(\frac1{t_k}\right).
\end{equation}

Using the same argument as in the proof of (ii), with $\hat z=z^*$, we have
\begin{equation}\label{eq_tqk}
	 t_{k+1}q_{k+1}
    =
    \frac{t_{k+1}-\eta}{t_k}\,t_kq_k-\eta r_k,
\end{equation}
where
\[
    q_k:=S(z_k-z^*),
    \qquad
    r_k:=
    D(z_{k+1}-\bar z_k)+\nabla h(\bar z_k)-\nabla h(z^*).
\]
From \eqref{eq:grado} and (a), we have
\begin{equation}\label{eq_ork}
	\|r_k\|= o\left(\frac1{t_k}\right).
\end{equation}

Let $
    \sigma_k:=1-\frac{t_{k+1}-\eta}{t_k}
    =
    \frac{\eta-(t_{k+1}-t_k)}{t_k}.
$
Since
$
    t_{k+1}-t_k
    =
    \frac{t_{k+1}^2-t_k^2}{t_{k+1}+t_k}
    \le
    \frac{\rho t_{k+1}}{t_{k+1}+t_k}
    <\rho,
$
and $0<\rho<\eta<1$, we have $\sigma_k\in(0,1)$. Moreover, by \eqref{eq:tkinf},
\[
    \sum_{k=1}^{+\infty}\sigma_k
    \ge
    (\eta-\rho)\sum_{k=1}^{+\infty}\frac1{t_k}
    =
    +\infty,
\]
and by \eqref{eq_ork},
$
    \frac{\|r_k\|}{\sigma_k}
    \le
    \frac{t_k\|r_k\|}{\eta-\rho}
    \to0.
$
From \eqref{eq_tqk}, we obtain
\[
 \|t_{k+1}q_{k+1}\|
    \le
   (1-\sigma_k)\|t_kq_k\|+\eta\sigma_k\frac{\|r_k\|}{\sigma_k}.
\]
Applying \Cref{le_bertsekas} with $q_k$ replaced by $\eta\|r_k\|/\sigma_k$, we obtain
$
   \|t_kq_k\|
   =
   t_k\|S(z_k-z^*)\|
   \to0.
$
That is,
\[
	\|S(z_k-z^*)\|=o\left(\frac1{t_k}\right).
\]
Combining this with \eqref{eq_onzk} and $F(z^*)=\nabla h(z^*)+Sz^*=0$, we obtain
\[
\|F(z_k)\|
=
\|\nabla h(z_k)+Sz_k\|
\le
\|\nabla h(z_k)-\nabla h(z^*)\|
+\|S(z_k-z^*)\|=
o\left(\frac1{t_k}\right).
\]
This proves (b), and the proof is complete.
\end{proof}

\begin{remark}
For the classical Nesterov rule \cite{Nesterov83}
$
    t_{k+1}=\frac{1+\sqrt{1+4t_k^2}}{2},
$
the Chambolle-Dossal rule \cite{ChambolleFista}
$
    t_k=\frac{k+\alpha-2}{\alpha-1},
$
and the Attouch-Cabot rule \cite{Attouch18siam} $
    t_k=\frac{k-1}{\alpha-1},
$ the condition \eqref{ass:tk} is satisfied. Hence \Cref{thm:disc-main} gives the $\bigO(1/k^2)$ convergence rate for the primal-dual gap and convergence of the iterates. Moreover, for the Chambolle-Dossal rule and the  Attouch-Cabot rule, one can choose parameters such that \eqref{ass:tk} holds with $\rho<1$. In this case, the improved $o(1/k^2)$ convergence rate for the primal-dual gap and the $o(1/k)$ convergence rate for the stationarity residual also follow. These discrete convergence results are consistent with the asymptotic convergence properties of the continuous dynamic \eqref{dyn:cont-xy} established in \Cref{thm:cont-est} and \Cref{thm:cont-main}.

Compared with accelerated primal-dual methods in \cite{CondatPre2026,HeCNSNS,Tran22,ChangJSC}, where convergence of the iterates and $\bigO(1/k^2)$ rates are obtained under semi-strong convexity assumptions, the present result is established under the merely convex-concave assumption. Compared with the accelerated bilinear saddle algorithm in \cite{DingCOAP}, where the $\bigO(1/k^2)$ gap estimate is proved and convergence of the iterates is established only in the case $\rho<1$, \Cref{thm:disc-main} proves convergence under the critical parameter regime $\eta=\rho$ as well, and further gives the little-$o$ primal-dual gap estimate in the noncritical case. In addition, the stationarity residual is improved to
$
    \|F(z_k)\|=o(1/t_k),
$
whereas the residual estimate in \cite{DingCOAP} is of order $o(1/\sqrt{t_k})$.
\end{remark}

\section{Numerical experiments}\label{sec:numerics}

In this section, we report two numerical examples to illustrate the behavior of the proposed fast primal-dual algorithm. The first example is a smooth convex-concave quadratic saddle problem without strong convexity. It is used to test the convergence results in the merely convex-concave setting. The second example is an image deblurring problem with quadratic gradient regularization, which leads to a partially strongly convex saddle formulation.

\subsection{Quadratic saddle problem}\label{subsec:num-rank-deficient}

We first consider a rank-deficient quadratic bilinear saddle problem
\begin{equation}\label{eq:num-rank-saddle}
    \min_{x\in\R^n}\max_{y\in\R^m}
    \ \cL(x,y)
    :=
    \frac12\|Bx-b\|^2+\langle Kx,y\rangle
    -\frac12\|Cy-d\|^2,
\end{equation}
where $B\in\R^{p\times n}$, $C\in\R^{q\times m}$, $b\in \R^n$, $d\in \R^m$ and $K\in\R^{m\times n}$. In the experiment, we take $n=1000$, $m=1200$, $p=800$ and $q=1000$, so that $B$ and $C$ are rank deficient. The coupling matrix $K$ is chosen as a dense matrix with decaying singular values and normalized by $\|K\|=1$. Hence \eqref{eq:num-rank-saddle} is a special case of \eqref{prob:saddle} with
$f(x)=\frac12\|Bx-b\|^2$ and
$g(y)=\frac12\|Cy-d\|^2$. Both $f$ and $g$ are convex and differentiable, with Lipschitz constants $L_f=\|B\|^2$ and $L_g=\|C\|^2$.
Since $B^\top B$ and $C^\top C$ are singular in general, neither $f$ nor $g$ is strongly convex. Thus this example directly corresponds to the merely convex-concave setting of \Cref{thm:disc-main}. The  saddle point $(x^*,y^*)$ is computed from the linear KKT system
\[
    \begin{pmatrix}
        B^\top B & K^\top\\
        -K & C^\top C
    \end{pmatrix}
    \begin{pmatrix}
        x^*\\
        y^*
    \end{pmatrix}
    =
    \begin{pmatrix}
        B^\top b\\
        C^\top d
    \end{pmatrix}.
\]
We compare the proposed algorithm, denoted by AL1 in the figures, with FPDA3 in \cite{DingCOAP} and the Chambolle-Pock method for the convex case \cite{ChambolleJMIV}, denoted by CP. For AL1 and FPDA3, we use the same Nesterov-type sequence
\[
    t_k=\frac{k+\alpha-2}{\alpha-1},
    \qquad
    \alpha=30.
\]
Then $\rho=2/(\alpha-1)$. The parameter $\eta$ in AL1 is chosen as $0.95$, and therefore the method is tested in the noncritical regime. The stepsizes of AL1 are chosen close to their theoretical upper bounds. The parameters of FPDA3 are chosen according to the admissible range in \cite{DingCOAP}. For the CP method, the stepsizes satisfy $\tau\sigma\|K\|^2<1$.

The numerical performance is evaluated by the primal-dual gap
\[
    \cG_{z^*}(z_k)=\cL(x_k,y^*)-\cL(x^*,y_k),
\]
the stationarity residual
\[
    \|F(z_k)\|
    =
    \left\|
    \begin{pmatrix}
        B^\top(Bx_k-b)+K^\top y_k\\
        C^\top(Cy_k-d)-Kx_k
    \end{pmatrix}
    \right\|,
\] 
and the increment $\|z_{k+1}-z_k\|$.  To illustrate the little-$o$ estimates, we also plot the scaled quantities
\[
    t_k^2\cG_{z^*}(z_k),
    \qquad
    t_k\|F(z_k)\|,
    \qquad
    t_k\|z_{k+1}-z_k\|.
\]
According to \Cref{thm:disc-main}, these quantities should decay to zero for AL1 in the noncritical regime.

\begin{figure}[htbp]
    \centering
    \includegraphics[width=0.95\linewidth]{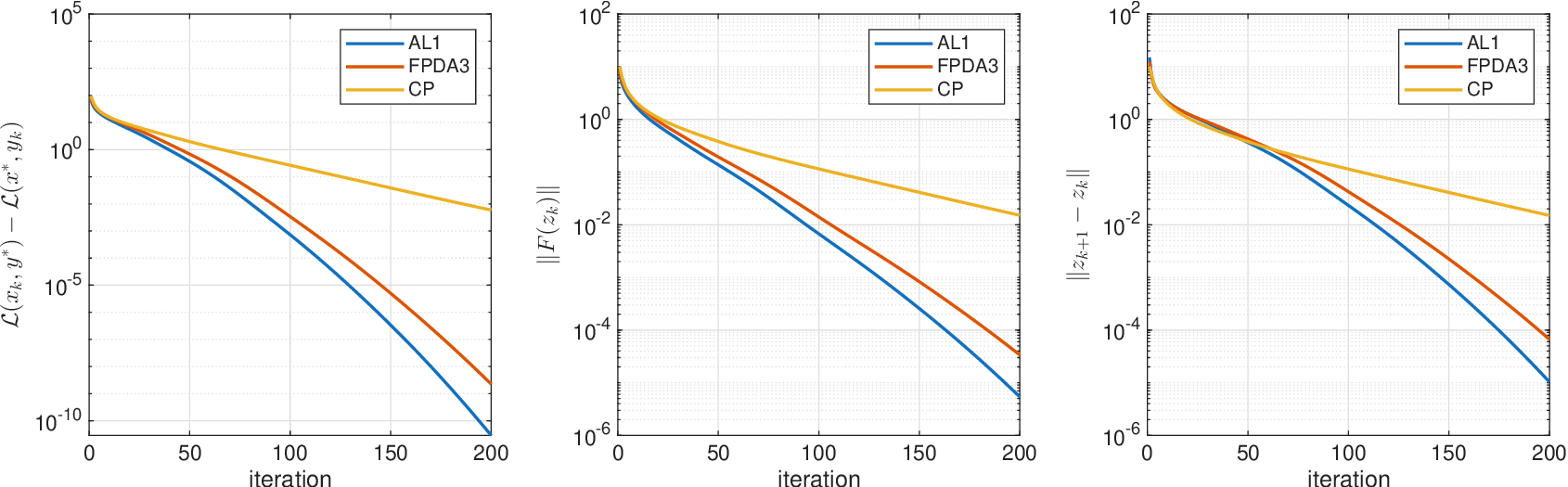}
    \caption{Comparison of the primal-dual gap, stationarity residual, and increment for  problem \eqref{eq:num-rank-saddle}.}
    \label{fig11}
\end{figure}

\begin{figure}[htbp]
    \centering
    \includegraphics[width=0.95\linewidth]{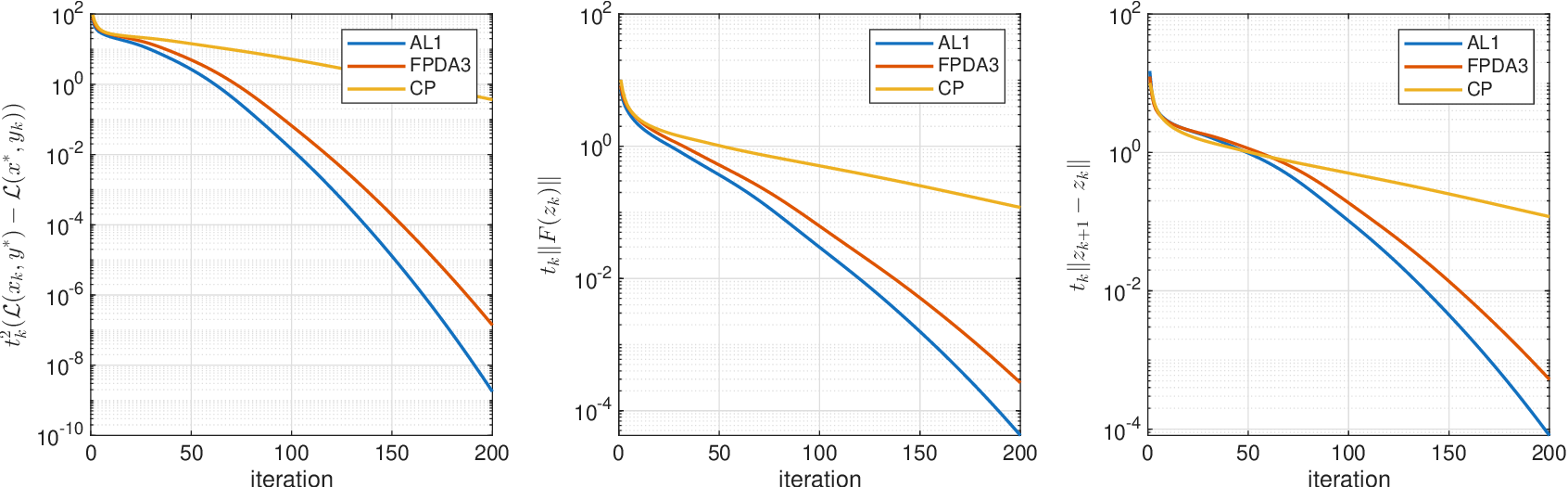}
    \caption{Scaled quantities corresponding to the refined little-$o$ estimates.}
    \label{fig12}
\end{figure}

As shown in \Cref{fig11}, AL1 gives stable decay of the primal-dual gap, the stationarity residual, and the increment, and is competitive with FPDA3 and the CP method. The scaled curves in \Cref{fig12} further support the little-$o$ behavior predicted by the theory.  

\subsection{Image deblurring} 

We next consider an image deblurring problem with quadratic gradient regularization. Let $x\in\R^n$ denote the unknown image and let $b\in\R^n$ be the blurred and noisy observation
\[
    b=Ax_{\rm true}+\varepsilon,
\]
where $A$ is a blurring operator, $x_{\rm true}$ is the original image, and $\varepsilon$ is additive noise. We consider
\[
    \min_{x\in\R^n}
    \frac12\|Ax-b\|^2+\frac{1}{2\mu}\|\mathcal Kx\|^2,
\]
where $\mathcal K$ is the discrete image-gradient operator and $\mu>0$. Using the identity
\[
    \frac{1}{2\mu}\|\mathcal Kx\|^2
    =
    \max_y\left\{
        \langle \mathcal Kx,y\rangle-\frac{\mu}{2}\|y\|^2
    \right\},
\]
the problem can be written as the bilinear saddle point problem
\begin{equation}\label{eq:num-saddle}
    \min_{x\in\R^n}\max_{y\in\R^m}
    \ \cL(x,y)
    :=
    \frac12\|Ax-b\|^2+\langle \mathcal Kx,y\rangle
    -\frac{\mu}{2}\|y\|^2.
\end{equation}
Thus $f(x)=\frac12\|Ax-b\|^2$, $K=\mathcal K$, and $g(y)=\frac{\mu}{2}\|y\|^2$. In this formulation, $g$ is strongly convex with parameter $\mu$, while $f$ may be ill-conditioned or rank deficient because of the blur operator. Hence \eqref{eq:num-saddle} is a partially strongly convex saddle problem.

For a two-dimensional image, the periodic discrete gradient operator is
\[
    \mathcal Kx=
    \begin{pmatrix}
        \mathcal K_hx\\
        \mathcal K_vx
    \end{pmatrix},
    \qquad
    (\mathcal K_hx)_{i,j}=x_{i+1,j}-x_{i,j},
    \qquad
    (\mathcal K_vx)_{i,j}=x_{i,j+1}-x_{i,j}.
\]
The adjoint operator $\mathcal K^\top$ is the corresponding negative discrete divergence. The reference saddle point is computed from the KKT system and is used to evaluate the primal-dual gap and the stationarity residual.

In the experiment, $A$ is chosen as a periodic Gaussian blur operator with kernel size $9$ and standard deviation $1.6$. The observation is corrupted by additive Gaussian noise with noise level $0.05$, and we take $\mu=1$. We compare AL1 with FPDA3 \cite{DingCOAP}, the accelerated Chambolle-Pock method \cite{ChambolleJMIV}, denoted by ACP, and the inertial accelerated primal-dual algorithm in \cite{HeCNSNS}, denoted by IAPDA. All methods are initialized from the same point.

\begin{figure}[htbp]
    \centering
    \includegraphics[width=0.95\linewidth]{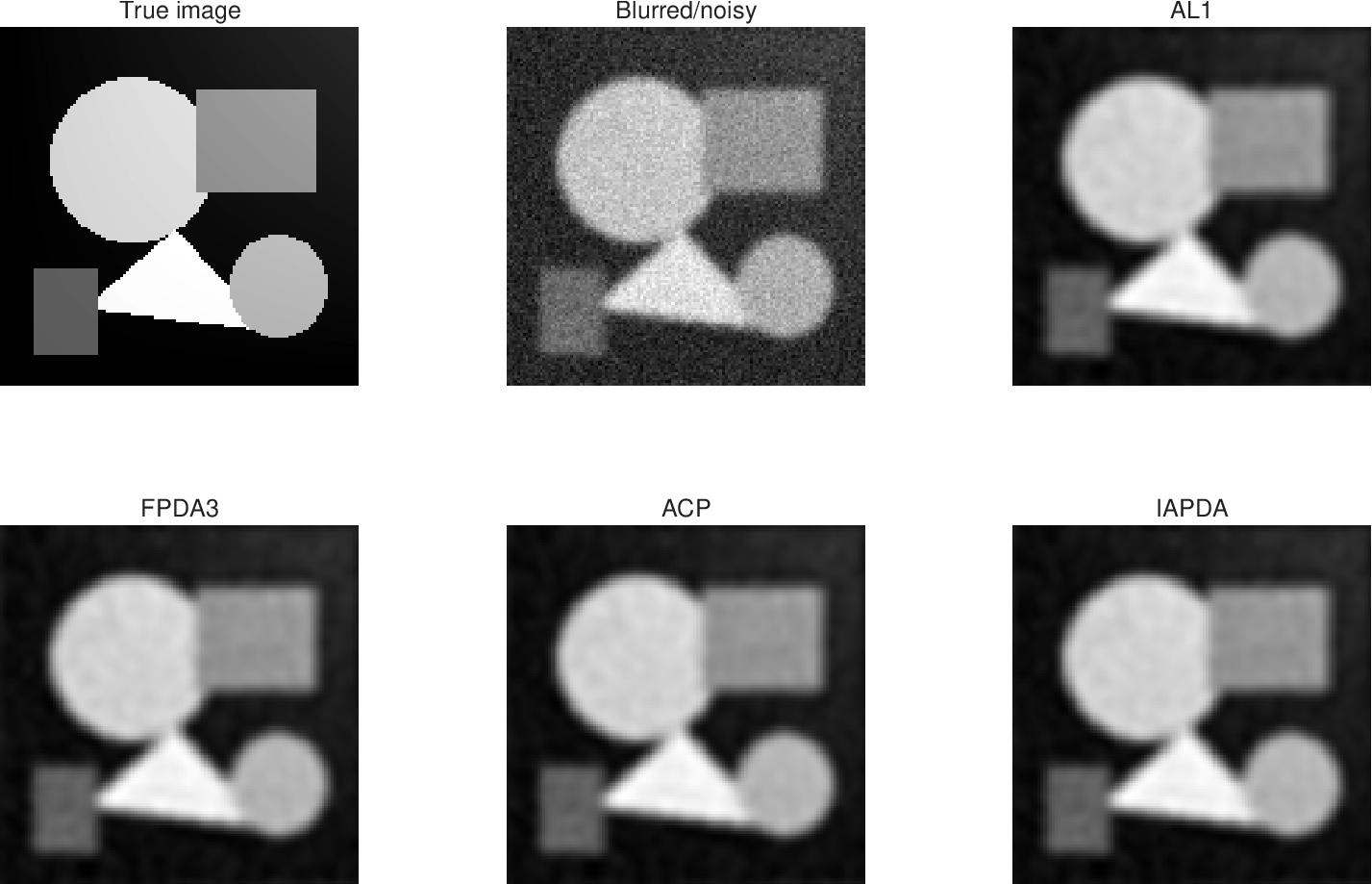}
    \caption{Image deblurring results. The six panels show the original image, the blurred and noisy observation, and the reconstructions obtained by AL1, FPDA3, ACP, and IAPDA.}
    \label{fig:recons}
\end{figure}

\begin{figure}[htbp]
    \centering
    \includegraphics[width=0.95\linewidth]{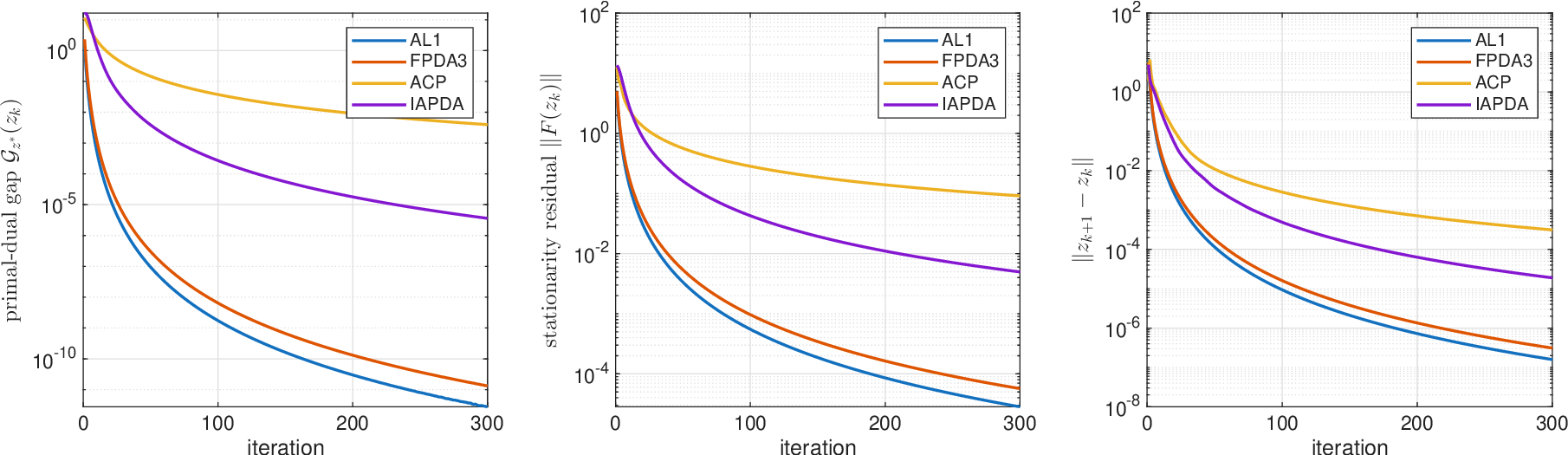}
    \caption{Comparison of the primal-dual gap, stationarity residual, and $\|z_{k+1}-z_k\|$ for the image deblurring problem.}
    \label{fig:conv}
\end{figure}

As shown in \Cref{fig:recons}, all four methods recover the main structures of the image from the blurred and noisy data. The convergence curves in \Cref{fig:conv} show that AL1 decreases the primal-dual gap and the stationarity residual faster in this test. These observations are consistent with the convergence results proved in this paper and indicate that the corrected bilinear coupling is also effective for the partially strongly convex image deblurring model.

\section{Conclusion}\label{sec:Conclusion}

In this paper, we studied Nesterov-type acceleration for bilinear convex-concave saddle point problems from both continuous and discrete viewpoints. For the continuous-time model, we established a Lyapunov estimate that yields the accelerated saddle-gap rate $\bigO(1/t^{2})$ and the velocity estimate $\bigO(1/t)$. We further proved convergence of the whole trajectory to a saddle point and obtained improved little-$o$ rates for the primal-dual gap, the velocity, and, under the smoothness assumption, the stationarity residual in the noncritical case. These results extend the improved asymptotic theory for Nesterov-type dynamics in convex minimization \cite{AttouchMp,MayTJM,AttouchSIOPT,Jang} and linearly constrained convex optimization \cite{BotJDE,HeTc,ZengTAC,HePDarxiv} to the bilinear saddle setting. On the discrete side, we derived a fast primal-dual algorithm through a structure-preserving discretization of the continuous dynamic. The resulting scheme preserves the skew-symmetric cancellation mechanism in the Lyapunov analysis and achieves the nonergodic rate $\bigO(1/t_k^{2})$ for the primal-dual gap. We also proved convergence of the generated sequence. When $\rho<\eta<1$, we further obtained the improved estimates $o(1/t_k^{2})$ for the gap and $o(1/t_k)$ for both the increment and the stationarity residual.
Compared with classical primal-dual splitting methods \cite{ChambolleJMIV}, the proposed method provides an accelerated nonergodic convergence guarantee. Compared with partially strongly convex accelerated methods \cite{CondatPre2026,HeCNSNS}, it does not rely on strong convexity. Compared with recent accelerated bilinear saddle algorithms \cite{DingCOAP}, it further establishes sequence convergence in the critical parameter regime and little-$o$ estimates in the noncritical regime. Numerical experiments on image deblurring support the theoretical results and show favorable decay of the primal-dual gap and the stationarity residual.

\end{document}